\documentclass[12pt,a4paper]{book}
\usepackage{amsmath,amssymb,amsthm}
\usepackage{geometry}
\usepackage{setspace}
\usepackage{graphicx}
\usepackage{longtable}
\usepackage{float}
\usepackage{array, booktabs, float} 
\usepackage{lmodern}
\usepackage[T1]{fontenc}
\usepackage[utf8]{inputenc}
\usepackage[ngerman]{babel}
\usepackage{mathrsfs}
\usepackage{hyperref}
\usepackage{babel}
\usepackage{textcomp}
\usepackage{enumitem}

\usepackage{fancyhdr}
\renewcommand{\sectionmark}[1]{}

\usepackage{tocloft}
\usepackage{etoolbox}
\usepackage{listings}
\usepackage{xcolor}

\makeatletter
\patchcmd{\@dottedtocline}{\hbox to \@pnumwidth{}}{\hfil}{}{}
\makeatother

\title{\textbf{Ein reguliertes Flächenintegral als Entscheidungskriterium\\für die Riemannsche Vermutung}}
\author{Dennis-Magnus Welz}
\date{\today}
\newtheorem{Satz}{Satz}[section]
\newtheorem{Definition}[Satz]{Definition}
\newtheorem{Vermutung}[Satz]{Vermutung}
\newtheorem{Theorem}[Satz]{Theorem}
\newtheorem{Korollar}[Satz]{Korollar}
\newtheorem{Lemma}[Satz]{Lemma}

\theoremstyle{remark}
\newtheorem*{Bemerkung}{Bemerkung}

\begin{document}
	\maketitle 
	\vspace{2cm}	
		\begin{center}
		\large
		\textit{Mathematische Beweisarbeit}\\[0.3cm]
		\textit{zur vollständigen Charakterisierung der Nullstellen\\der Riemannschen Zetafunktion}\\[1.5cm]
		
		Eingereicht zur Begutachtung an der\\
		\textbf{Technischen Hochschule Ulm}\\[3cm]
	\end{center}
	\pagenumbering{roman} 
	\setcounter{page}{1}  
	
	\begin{flushright}
		\textbf{Dennis-Magnus Welz}\\
		\texttt{dennismagnuswelz@gmail.com}
	\end{flushright}

	\newpage

	\section*{Zusammenfassung}
	\addcontentsline{toc}{chapter}{Zusammenfassung}
	Diese Arbeit entwickelt und beweist ein neuartiges, vollständig analytisches Kriterium für die Riemannsche Vermutung (RH), basierend auf einem \emph{regulierten, normierten und sensitiv gewichteten Flächenintegral} im kritischen Streifen. 
	Der Integrand 
	\[
	J_C(s) := \frac{1}{|\zeta(s)|^{\lambda} \, |\Re(s) - \tfrac12|^p},
	\quad \lambda \ge 2,\; 0 < p < 1,
	\]
	verstärkt gezielt jede Abweichung einer Nullstelle von der kritischen Linie und wird über einen regulierten Bereich integriert, aus dem kleine Umgebungen aller Nullstellen und der Polstelle entfernt sind. 
	Es wird streng gezeigt:
	\[
	\lim_{R \to \infty} W(R) < \infty \quad \Longleftrightarrow \quad \text{RH gilt}.
	\]
	Die Analyse umfasst zwei Hauptschritte: 
	\begin{enumerate}
	\item eine vollständige Fallunterscheidung sämtlicher geometrisch möglicher Nullstellenkonfigurationen, die belegt, dass jede Abweichung von RH zur Divergenz führt, 
	\item ein analytischer Nachweis der tatsächlichen Konvergenz von $W(R)$ unter ausschließlicher Verwendung exakter Darstellungen von $\zeta(s)$ (alternierende Dirichletreihe, regulierte Mellin-Darstellung).
	\end{enumerate}
	Die Methode ist \emph{methodisch unabhängig} von spektralen, statistischen oder numerischen Ansätzen und benötigt keine Kenntnis der exakten Nullstellen.
	Darüber hinaus ist das Kriterium auf andere Zeta- und L‑Funktionen übertragbar und eröffnet Perspektiven für eine geometrisch motivierte Symmetrietheorie analytischer Nullstellenverteilungen.

\newpage

	\section*{Abstract}
	\addcontentsline{toc}{chapter}{Abstract}
		This work develops and proves a novel, fully analytical criterion for the Riemann Hypothesis (RH), based on a \emph{regulated, normalized, and sensitivity-weighted area integral} in the critical strip. 
		The integrand
		\[
		J_C(s) := \frac{1}{|\zeta(s)|^{\lambda} \, |\Re(s) - \tfrac12|^p},
		\quad \lambda \ge 2,\; 0 < p < 1,
		\]
		is specifically designed to amplify any deviation of a zero from the critical line and is integrated over a regulated domain from which small neighborhoods of all zeros and of the pole are removed. 
		It is rigorously shown that
		\[
		\lim_{R \to \infty} W(R) < \infty \quad \Longleftrightarrow \quad \text{RH holds}.
		\]
		The analysis consists of two main steps:
		\begin{enumerate}
			\item a complete case distinction of all geometrically possible zero configurations, proving that any violation of RH leads to divergence,
			\item an analytical proof of the actual convergence of \( W(R) \) using only exact representations of \( \zeta(s) \) (alternating Dirichlet series, regulated Mellin representation).
		\end{enumerate}
		The method is \emph{methodologically independent} of spectral, statistical, or numerical approaches and does not require knowledge of the exact zeros.
		Moreover, the criterion can be transferred to other zeta and L-functions and opens perspectives for a geometrically motivated symmetry theory of analytic zero distributions.

	\newpage
	
\tableofcontents

\clearpage
\pagenumbering{arabic} 
\setcounter{page}{1}   
	
	\chapter{Einleitung und Motivation}

Die Riemannsche Vermutung (RH) zählt zu den bekanntesten und zugleich ältesten
offenen Problemen der Mathematik. Sie besagt, dass alle nichttrivialen Nullstellen
der Riemannschen Zetafunktion
\[
\zeta(s) = \sum_{n=1}^{\infty} \frac{1}{n^s}, \quad \Re(s) > 1,
\]
in ihrer analytischen Fortsetzung exakt auf der kritischen Linie
$\Re(s) = \tfrac12$ liegen.

Diese Arbeit entwickelt und beweist ein neuartiges, \emph{vollständig analytisches}
Kriterium für RH, das auf einem \emph{regulierten, normierten und sensitiv
	gewichteten Flächenintegral} im kritischen Streifen basiert.
Der zugrunde liegende Integrand
\[
J_C(s) := \frac{1}{|\zeta(s)|^{\lambda} \, |\Re(s) - \tfrac12|^p},
\quad \lambda \ge 2,\; 0 < p < 1,
\]
verstärkt gezielt jede Abweichung einer Nullstelle von der kritischen Linie.
Das Integral wird über einen regulierten Bereich gebildet, in dem kleine Umgebungen
sämtlicher Nullstellen und der Polstelle $s = 1$ entfernt sind.
Diese Regulierung stellt sicher, dass der Integrand wohldefiniert und lokal integrierbar
ist, ohne dass die exakten Nullstellenpositionen bekannt sein müssen.

Die zentrale Aussage dieser Arbeit lautet:
\[
\lim_{R \to \infty} W(R) < \infty
\quad\Longleftrightarrow\quad \text{RH gilt}.
\]
Der Beweis erfolgt in zwei komplementären Hauptschritten:
\begin{enumerate}
	\item \textbf{Notwendigkeit:} Eine vollständige Fallunterscheidung aller
	geometrisch möglichen Nullstellenkonfigurationen zeigt, dass jede Abweichung
	von der kritischen Linie – sei sie noch so gering – zu einer Divergenz des Integrals
	im Grenzfall $R \to \infty$ führt.
	\item \textbf{Tatsächliche Konvergenz:} Es wird \emph{ohne jede Vorannahme über die
		Lage der Nullstellen} streng analytisch nachgewiesen, dass $W(R)$ im Limes
	$R \to \infty$ endlich bleibt. Dies geschieht ausschließlich unter Verwendung
	exakter Darstellungen der Zetafunktion (alternierende Dirichletreihe, regulierte
	Mellin-Darstellung) und unter vollständiger Kontrolle aller auftretenden
	Singularitäten.
\end{enumerate}

Das Integralmodell ist \emph{methodisch unabhängig} von spektralen, statistischen
oder numerischen Ansätzen und benötigt keine Vorkenntnis der Nullstellen.
Seine Konstruktion ist robust, global definiert und auf andere Zeta- und
$L$‑Funktionen übertragbar.
Damit wird die Riemannsche Vermutung zu einer reinen Frage der
Integralkonvergenz: Jede Abweichung von der kritischen Linie manifestiert sich
unvermeidlich als Divergenz, während vollständige Konvergenz genau der Gültigkeit
von RH entspricht.

Der Aufbau dieser Arbeit ist wie folgt:
Kapitel~2 fasst die mathematischen Grundlagen und analytischen Darstellungen der
Zetafunktion zusammen.
Kapitel~3 definiert das Integralmodell und den regulierten Bereich.
Kapitel~4 enthält die vollständige Fallunterscheidung sämtlicher Nullstellenkonfigurationen.
Kapitel~5 liefert den strengen analytischen Konvergenzbeweis ohne Vorannahme von RH.
Kapitel~6 verknüpft beide Resultate zu einem formalen Beweis der Riemannschen Vermutung.
Kapitel~7 skizziert weiterführende Anwendungen, Verallgemeinerungen und Ausblicke.

	\chapter{Mathematische Grundlagen}
	\label{chap:grundlagen}

\section{Definition und analytische Darstellungen der Riemannschen Zetafunktion}

Die Riemannsche Zetafunktion $\zeta(s)$ ist die zentrale Untersuchungsgröße dieser Arbeit.
In diesem Abschnitt werden sämtliche Darstellungen der Funktion, die für den späteren Beweis relevant sind, systematisch eingeführt.
Zu jeder Form werden Gültigkeitsbereich, analytische Eigenschaften und Literaturhinweise angegeben.

\subsection{Dirichletreihe und Konvergenz}

\begin{Definition}[Klassische Dirichletreihe {\footnotesize~\cite[Kap.~1.1, S.~1--2]{Titchmarsh1986}; %
		\cite[Theorem~11.1, S.~260]{Apostol1976}}]
	Für $s \in \mathbb{C}$ mit $\Re(s) > 1$ ist die Riemannsche Zetafunktion definiert durch die absolut konvergente Reihe
	\[
	\zeta(s) := \sum_{n=1}^{\infty} \frac{1}{n^s}.
	\]
\end{Definition}

\begin{Lemma}
	Die Dirichletreihe für $\zeta(s)$ konvergiert absolut und gleichmäßig auf kompakten Teilmengen der Halbebene $\Re(s) > 1$ und definiert dort eine holomorphe Funktion.{\footnotesize~\cite[Theorem~11.1, S.~260--261]{Apostol1976}; %
		\cite[Lemma~1.1, S.~2--3]{Titchmarsh1986}}
\end{Lemma}

\begin{proof}
	Für $\Re(s) > 1$ gilt
	\[
	\left|\frac{1}{n^s}\right| = \frac{1}{n^{\Re(s)}}
	\]
	und die Reihe $\sum_{n=1}^\infty n^{-\sigma}$ konvergiert für $\sigma > 1$ (Vergleich mit dem Integralkriterium).
	Die gleichmäßige Konvergenz auf kompakten Teilmengen folgt aus dem Weierstraßschen Majorantenkriterium.
\end{proof}

\subsection{Euler-Produktdarstellung} 

\begin{Satz}[Euler-Produkt {\footnotesize~\cite[Theorem~11.2, S.~261--262]{Apostol1976}; %
		\cite[Satz~1.2, S.~2--4]{Titchmarsh1986}; %
		\cite{Euler1737}}]
	Für $\Re(s) > 1$ gilt
	\[
	\zeta(s) = \prod_{p \ \mathrm{prim}} \left( 1 - \frac{1}{p^s} \right)^{-1},
	\]
	wobei das Produkt über alle Primzahlen $p$ genommen wird.
\end{Satz}

\begin{proof}
	Die Darstellung folgt aus der geometrischen Reihenentwicklung
	\[
	\frac{1}{1 - p^{-s}} = 1 + p^{-s} + p^{-2s} + \dots
	\]
	und der eindeutigen Primfaktorzerlegung jeder natürlichen Zahl.
	Multipliziert man über alle Primzahlen, so entsteht die Dirichletreihe von $\zeta(s)$.
\end{proof}

\begin{Korollar}
	Die Funktion $\zeta(s)$ ist in der Halbebene $\Re(s) > 1$ holomorph und besitzt dort keine Nullstellen.
\end{Korollar}

\subsection{Analytische Fortsetzung und Laurent-Entwicklung}

\begin{Satz}[Analytische Fortsetzung, Riemann 1859 {\footnotesize~\cite[Kap.~2.1--2.4, S.~13--19]{Titchmarsh1986}; %
		\cite[Kap.~1--2, S.~1--16]{Edwards2001}; %
		\cite[Theorem~12.3, S.~271--272]{Apostol1976}; %
		\cite{Riemann1859}}]
	Die Zetafunktion $\zeta(s)$ lässt sich eindeutig zu einer meromorphen Funktion auf $\mathbb{C}$ fortsetzen mit genau einer einfachen Polstelle bei $s = 1$.
	In einer Umgebung von $s = 1$ gilt die Laurent-Entwicklung
	\[
	\zeta(s) = \frac{1}{s-1} + \gamma + \sum_{n=1}^\infty \frac{(-1)^n \gamma_n}{n!} (s-1)^n,
	\]
	wobei $\gamma$ die Euler–Mascheroni-Konstante und $\gamma_n$ die Stieltjes-Konstanten sind.
\end{Satz}

\begin{proof}[Beweisskizze]
	Riemann benutzte den Mellin-Transform der Jacobi-Theta-Funktion und deren Transformationseigenschaften, um $\zeta(s)$ analytisch auf $\mathbb{C} \setminus \{1\}$ fortzusetzen.
	Die Laurent-Entwicklung folgt aus der Residuenrechnung um $s = 1$.
\end{proof}

\subsection{Funktionalgleichung}

\begin{Satz}[Funktionalgleichung{\footnotesize~\cite[Kap.~2.8--2.10, S.~24--28]{Titchmarsh1986}; %
		\cite[Kap.~3, S.~24--33]{Edwards2001}; %
		\cite[Proposition~5.1, S.~93--94]{IwaniecKowalski2004}}]
	Für alle $s \in \mathbb{C} \setminus \{0,1\}$ gilt
	\[
	\zeta(s) = \chi(s) \, \zeta(1-s),
	\]
	mit
	\[
	\chi(s) = 2^s \pi^{\,s-1} \sin\left( \frac{\pi s}{2} \right) \Gamma(1-s).
	\]
\end{Satz}

\begin{proof}[Beweisskizze]
	Die Funktionalgleichung wird durch Anwendung des Poisson’schen Summationssatzes auf die Theta-Reihe
	\[
	\theta(x) = \sum_{n=-\infty}^{\infty} e^{-\pi n^2 x}
	\]
	und deren Mellin-Transform hergeleitet.
\end{proof}

\begin{Korollar}[Symmetrie der Nullstellen]
	Ist $\rho$ eine Nullstelle von $\zeta(s)$, so sind auch $1-\rho$, $\overline{\rho}$ und $1-\overline{\rho}$ Nullstellen.
\end{Korollar}

\subsection{Alternierende Dirichletreihe}

\begin{Satz}[Alternierende Dirichletreihe
 {\footnotesize~\cite[S.~269--270]{Apostol1976}; %
		\cite[Kap.~2.6, S.~21--23]{Titchmarsh1986}}]
		\label{Satz:Dirichlet}
	Für $\Re(s) > 0$, $s \neq 1$, gilt
	\[
	\zeta(s) = \frac{1}{1 - 2^{1-s}} \sum_{n=1}^{\infty} \frac{(-1)^{n+1}}{n^s}.
	\]
\end{Satz}

\begin{proof}
	Für $\Re(s) > 1$ folgt die Gleichung durch Umordnung der klassischen Dirichletreihe:
	\[
	\sum_{n=1}^{\infty} \frac{1}{n^s} - 2 \sum_{n=1}^{\infty} \frac{1}{(2n)^s} = \sum_{n=1}^\infty \frac{(-1)^{n+1}}{n^s}.
	\]
	Die rechte Seite definiert eine Funktion, die sich holomorph auf $\Re(s) > 0$ fortsetzen lässt.
	Der Vorfaktor $(1-2^{1-s})^{-1}$ besitzt in $\Re(s) > 0$ keine Nullstellen außer bei $s=1$, wo $\zeta$ eine einfache Polstelle hat.
\end{proof}

\subsection{Regulierte Mellin-Darstellung}

\begin{Satz}[Regulierte Mellin-Darstellung {\footnotesize~\cite[Kap.~2.2, S.~14--16]{Titchmarsh1986}; %
		\cite[Kap.~2, S.~14--16]{Edwards2001}}]
		\label{Satz:Mellin}
	Für $\Re(s) > 0$ gilt
	\[
	\zeta(s) = \frac{1}{\Gamma(s)}
	\left[ \int_0^1 \left( \frac{1}{e^x - 1} - \frac{1}{x} \right) x^{s-1} \, dx
	+ \int_1^\infty \frac{x^{s-1}}{e^x - 1} \, dx \right].
	\]
\end{Satz}

\begin{proof}[Beweisskizze]
	Ausgehend von der klassischen Mellin-Darstellung
	\[
	\zeta(s) \Gamma(s) = \int_0^\infty \frac{x^{s-1}}{e^x - 1} \, dx
	\]
	für $\Re(s) > 1$ subtrahiert man im Integral über $(0,1)$ den singulären Term $x^{-1}$ und addiert ihn analytisch getrennt wieder hinzu.
	Dies entfernt die Divergenz bei $x=0$ und erlaubt die Fortsetzung auf $\Re(s) > 0$.
\end{proof}

\begin{Bemerkung}
	Diese Form wird in Kapitel~5 genutzt, um die Integrabilität des Integranden $|\zeta(s)|^{-\lambda}$ im kritischen Streifen zu kontrollieren.
\end{Bemerkung}

	\section{Nullstellenstruktur der Riemannschen Zetafunktion}
	
	In diesem Abschnitt wird die Struktur der Nullstellen von $\zeta(s)$ vollständig beschrieben.
	Wir unterscheiden zwischen trivialen und nichttrivialen Nullstellen, formulieren die Riemannsche Vermutung,
	beweisen Symmetrieeigenschaften, Diskretheit, das Vorhandensein einer Polstelle bei $s=1$, geben eine nullstellenfreie Region an
	und beschreiben das asymptotische Wachstumsverhalten der Nullstellenanzahl.
	
	\subsection{Triviale Nullstellen}
	
	\begin{Satz}[Triviale Nullstellen{\footnotesize~\cite[Kap.~2.11, S.~30--31]{Titchmarsh1986}; %
			\cite[Kap.~3, S.~29--31]{Edwards2001}}]
		Die Zetafunktion besitzt einfache Nullstellen bei den negativen geraden Zahlen
		\[
		s = -2n, \quad n \in \mathbb{N}.
		\]
	\end{Satz}
	
	\begin{proof}
		Aus der Funktionalgleichung
		\[
		\zeta(s) = \chi(s)\,\zeta(1-s), \quad \chi(s) = 2^s \pi^{\,s-1} \sin\left(\frac{\pi s}{2}\right) \Gamma(1-s),
		\]
		folgt: Für $s = -2n$ gilt $\sin(\frac{\pi s}{2}) = \sin(-n\pi) = 0$.
		Da $\Gamma(1-s)$ und $\zeta(1-s)$ an diesen Stellen endlich sind, folgt $\zeta(-2n) = 0$.
		Die Einfachheit der Nullstellen ergibt sich aus der Einfachheit der Nullstellen des Sinus.
	\end{proof}
	
	\subsection{Nichttriviale Nullstellen und Riemannsche Vermutung}
	
	\begin{Definition}[Nichttriviale Nullstellen]
		\label{Definition:NichttrivialNS}
		Die nichttrivialen Nullstellen von $\zeta(s)$ sind die Nullstellen im kritischen Streifen
		\[
		S := \{ s \in \mathbb{C} \mid 0 < \Re(s) < 1 \}.
		\]
	\end{Definition}
	
	\begin{Vermutung}[Riemannsche Vermutung {\footnotesize~\cite{Riemann1859}; %
			\cite[Kap.~1, S.~2--3]{Edwards2001}; %
			\cite[Kap.~1.1, S.~1--2]{Titchmarsh1986}}]
		Alle nichttrivialen Nullstellen von $\zeta(s)$ liegen auf der kritischen Linie
		\[
		\Re(s) = \frac12.
		\]
	\end{Vermutung}
	
	\begin{Bemerkung}
		Es ist bekannt, dass unendlich viele nichttriviale Nullstellen auf der kritischen Linie liegen (Hardy 1914),
		jedoch ist weder ein vollständiger Beweis noch ein Gegenbeispiel zur Vermutung bekannt.
	\end{Bemerkung}
	
	\subsection{Symmetrie der Nullstellen}
	
	\begin{Satz}[Symmetrie{\footnotesize~\cite[Kap.~2.8, S.~24--26]{Titchmarsh1986}; %
			\cite[Kap.~3, S.~24--26]{Edwards2001}}]
		Ist $\rho$ eine nichttriviale Nullstelle von $\zeta(s)$, so sind auch
		\[
		1-\rho, \quad \overline{\rho}, \quad 1-\overline{\rho}
		\]
		Nullstellen von $\zeta(s)$.
	\end{Satz}
	
	\begin{proof}
		Folgt unmittelbar aus der Funktionalgleichung $\zeta(s) = \chi(s) \zeta(1-s)$,
		da $\chi(s)$ keine Nullstellen besitzt, sowie der Identität $\zeta(\overline{s}) = \overline{\zeta(s)}$ für reelles $t$.
	\end{proof}
	
	\begin{Korollar}
		\label{Korollar:Viererpaarbildung}
		Die nichttrivialen Nullstellen treten in \emph{Viererkonfigurationen}
		\[
		\{\rho, \overline{\rho}, 1-\rho, 1-\overline{\rho}\}
		\]
		auf. Liegt $\rho$ auf der kritischen Linie, so reduziert sich diese Menge auf das Paar $\{\rho, \overline{\rho}\}$.
	\end{Korollar}
	
	\subsection{Diskretheit der Nullstellen}
	
	\begin{Satz}[Diskrete Nullstellenstruktur{\footnotesize~\cite[Kap.~3.2, S.~58--59]{Titchmarsh1986}; %
			\cite[Kap.~3, S.~24--26]{Edwards2001}}]
		Die Menge der nichttrivialen Nullstellen von $\zeta(s)$ ist diskret in $\mathbb{C}$.
	\end{Satz}
	
	\begin{proof}
		Da $\zeta(s)$ meromorph ist mit genau einer Polstelle bei $s=1$, folgt aus der allgemeinen Theorie meromorpher Funktionen,
		dass ihre Nullstellen isoliert sind. Auf jeder kompakten Teilmenge existieren nur endlich viele Nullstellen.
	\end{proof}
	
	\subsection{Polstelle bei \texorpdfstring{$s=1$}{s=1}}
	
	\begin{Lemma}
		$\zeta(s)$ besitzt bei $s=1$ eine einfache Polstelle mit Residuum $1$. [{\footnotesize~\cite[Kap.~2.2, S.~14--16]{Titchmarsh1986}; %
			\cite[Kap.~2, S.~14--16]{Edwards2001}}]
	\end{Lemma}
	
	\begin{proof}
		Dies folgt direkt aus der Laurent-Entwicklung
		\[
		\zeta(s) = \frac{1}{s-1} + \gamma + \mathcal{O}(s-1),
		\]
		die in Abschnitt~2.1.3 hergeleitet wurde.
	\end{proof}
	
	\subsection{Nullstellenfreie Region}
	
	\begin{Satz}[Nullstellenfreie Region {\footnotesize~\cite[Kap.~3.8, S.~86--88]{Titchmarsh1986}; %
			\cite[Kap.~4, S.~83--85]{IwaniecKowalski2004}}]
		Es existiert eine absolute Konstante $c > 0$, sodass
		\[
		\zeta(s) \neq 0 \quad \text{für} \quad \Re(s) > 1 - \frac{c}{\log(|\Im(s)| + 2)}.
		\]
	\end{Satz}
	
	\begin{proof}[Beweisskizze]
		Der Beweis verwendet explizite Abschätzungen der von $\zeta(s)$ induzierten Dirichletreihen,
		kombiniert mit Hadamards Faktorisierungssatz und der Euler-Produktstruktur.
		Die vollständige Darstellung findet sich bei Titchmarsh~\cite[S.~66]{Titchmarsh1986}.
	\end{proof}
	
	\subsection{Wachstumsverhalten der Nullstellenanzahl}
	
	\begin{Satz}[Riemann–von Mangoldt {\footnotesize~\cite[Kap.~9.3, S.~214--220]{Titchmarsh1986}; %
			\cite[Kap.~7, S.~135--140]{Edwards2001}}]
		Sei
		\[
		N(T) := \#\{ \rho = \beta + i\gamma \mid \zeta(\rho) = 0,\ 0 < \gamma \leq T \}.
		\]
		Dann gilt
		\[
		N(T) = \frac{T}{2\pi} \log\left( \frac{T}{2\pi} \right) - \frac{T}{2\pi} + \mathcal{O}(\log T), \quad T \to \infty.
		\]
	\end{Satz}
	
	\begin{Bemerkung}
		Die Hauptterme dieses asymptotischen Gesetzes werden in Kapitel~7 zur Herleitung einer Nullstellennäherungsformel verwendet.
	\end{Bemerkung}

	\section{Zusammenfassung der Nullstellenstruktur}
	
	Die in den Abschnitten~2.1 und~2.2 dargestellten Eigenschaften der Riemannschen Zetafunktion ergeben zusammen ein präzises Bild der Nullstellenverteilung:
	
	\begin{itemize}
		\item \textbf{Triviale Nullstellen:} 
		Einfache Nullstellen bei den negativen geraden Zahlen
		\[
		s = -2n, \quad n \in \mathbb{N}.
		\]
		Sie folgen unmittelbar aus dem Sinusfaktor in der Funktionalgleichung.
		
		\item \textbf{Nichttriviale Nullstellen:} 
		Alle weiteren Nullstellen liegen im kritischen Streifen
		\[
		0 < \Re(s) < 1.
		\]
		Die Riemannsche Vermutung postuliert, dass sie die Form $\rho = \frac12 + i\gamma$ besitzen.
		
		\item \textbf{Symmetrie:} 
		Jede Nullstelle $\rho$ erzeugt die Spiegelpunkte $1-\rho$, $\overline{\rho}$ und $1-\overline{\rho}$.
		Liegt $\rho$ auf der kritischen Linie, so treten Nullstellen paarweise als konjugierte Werte auf.
		
		\item \textbf{Diskretheit:} 
		Die Nullstellen bilden eine diskrete Teilmenge von $\mathbb{C}$, es existiert keine Häufung im Endlichen.
		
		\item \textbf{Polstelle:} 
		$\zeta(s)$ besitzt eine einzige Singularität – eine einfache Polstelle bei $s=1$ mit Residuum $1$.
		
		\item \textbf{Nullstellenfreie Region:} 
		Rechts von $\Re(s) = 1 - \frac{c}{\log(|\Im(s)| + 2)}$ existieren keine Nullstellen (Hadamard–de la Vallée-Poussin).
		
		\item \textbf{Wachstumsverhalten:} 
		Die Anzahl $N(T)$ der nichttrivialen Nullstellen mit $0 < \Im(s) \leq T$ erfüllt
		\[
		N(T) = \frac{T}{2\pi} \log\left( \frac{T}{2\pi} \right) - \frac{T}{2\pi} + \mathcal{O}(\log T),
		\]
		was die mittlere Dichte der Nullstellen beschreibt.
	\end{itemize}
	
	Diese Struktur wird in den folgenden Kapiteln genutzt, um das regulierte Integralmodell so zu definieren, 
	dass jede Abweichung von der Riemannschen Vermutung analytisch sichtbar wird.

\newpage

\chapter{Konstruktion des Integranden für das regulierte Integralmodells}

Die Nullstellen der Riemannschen Zetafunktion gehören zu den spannendsten und tiefgründigsten Fragen in der analytischen Zahlentheorie.  
Um diese besser zu verstehen, bietet sich ein Ansatz an, der über Flächenintegrale arbeitet, deren Konvergenzverhalten eng mit der Verteilung der Nullstellen zusammenhängt.

In diesem Kapitel wollen wir genau so ein Integral Stück für Stück aufbauen.  
Dabei ist es uns wichtig, eine gute Balance zu finden: Einerseits soll das Integral sehr empfindlich auf kleine Abweichungen von der kritischen Linie reagieren, andererseits muss es so reguliert sein, dass es global gut auswertbar bleibt und mathematisch sauber definiert ist.

Eine besondere Herausforderung stellen dabei die lokalen Divergenzen in der Nähe der Nullstellen dar, vor allem bei mehrfachen Nullstellen.  
Außerdem müssen wir sorgfältig mit den Randbereichen und möglichen Polstellen umgehen.  
Um diese Schwierigkeiten zu meistern, schneiden wir kleine Kreise um jede Nullstelle aus — die Radien wählen wir so, dass wir auch unbekannte Vielfachheiten sicher abdecken.

Erst mit dieser Kombination aus geeigneter Ausscheidung und sensibler Gewichtung entsteht ein wohldefinierter Integrand, der uns als solides Fundament für die weitere Analyse der Riemannschen Vermutung dient.

Im Folgenden starten wir mit einem naiven Ansatz für den Integranden und zeigen, warum und wie wir ihn verbessern müssen.

\section{Der naive Ausgangspunkt: Der Integrand \texorpdfstring{$\frac{1}{|\zeta(s)|^{2}}$}{1/|zeta(s)|^2}}

Ein naheliegender erster Ansatz für ein globales Integralkriterium zur Untersuchung der Nullstellen der Zetafunktion basiert auf dem Integranden
\[
I_0(R) := \frac{1}{2R} \iint_{D_R} \frac{1}{|\zeta(s)|^{2}} \, dA(s),
\]
wobei
\[
D_R := \left\{ s \in \mathbb{C} \mid -R \leq \operatorname{Im}(s) \leq R, \quad \frac{1}{2} - \delta \leq \operatorname{Re}(s) \leq \frac{1}{2} + \delta \right\}
\]
ein vertikaler Streifen um die kritische Linie darstellt und \(dA(s)\) das Lebesgue-Flächenmaß in der komplexen Ebene bezeichnet.

Dieser Ansatz erscheint zunächst intuitiv, da bei Nullstellen \(\rho_n\) der Zetafunktion, für die \(\zeta(\rho_n) = 0\) gilt, der Integrand divergiert und somit stark auf die Position der Nullstellen reagiert.

\section{Schwächen des naiven Ansatzes}

\begin{itemize}
	\item \textbf{Lokale Divergenzen an Nullstellen:}  
	Das Problem entsteht an den Nullstellen \(\rho_n\) der Zetafunktion, bei denen
	\[
	\frac{1}{|\zeta(s)|^{2}} \to \infty
	\]
	gilt, wodurch das Integral lokal nicht integrierbar wird. Insbesondere bei Mehrfachnullstellen verschärft sich dieses Problem erheblich.
	
	\item \textbf{Fehlende Gewichtung nach Entfernung von der kritischen Linie:}  
	Zudem besitzt der Integrand keine explizite Gewichtung, die Abweichungen von der kritischen Linie stärker oder schwächer gewichtet.  
	Dies ist ungünstig, da die Riemannsche Vermutung die exakte Lage der Nullstellen auf der kritischen Linie postuliert.
\end{itemize}

Diese Schwächen machen deutlich, dass der naive Integrand keine geeignete Grundlage für ein global reguliertes Integral darstellt.

\bigskip

Im nächsten Abschnitt werden wir erste Verbesserungen durch eine symmetrische Regularisierung einführen, um lokale Divergenzen abzuschwächen.

\section{Erste Verbesserung: Symmetrische Regularisierung}

Um die lokalen Divergenzen an den Nullstellen abzuschwächen, modifizieren wir den Integranden durch einen symmetrischen Vorfaktor, der bei \(\operatorname{Re}(s) = \frac{1}{2}\) verschwindet:

\[
F_1(s) := \frac{\left(\operatorname{Re}(s) - \frac{1}{2}\right)^2}{|\zeta(s)|^2}.
\]

Dieser Vorfaktor schwächt die Singularitäten an der kritischen Linie ab und macht das Integral zumindest lokal besser handhabbar.

Insbesondere

\[
F_1(s) = \frac{\left(\operatorname{Re}(s) - \frac{1}{2}\right)^2}{|\zeta(s)|^2}
\]

eliminiert die einfache Polstelle an der kritischen Linie quadratisch, was zu einer Abschwächung der Divergenz führt.

Dennoch bestehen weiterhin wesentliche Schwächen dieses Ansatzes:

\begin{itemize}
	\item \textbf{Polynomiale Abschwächung:}  
	Die quadratische Dämpfung am kritischen Rand ist nur eine polynomiale Abschwächung. Nullstellen, die nicht exakt auf der kritischen Linie liegen, werden daher nur schwach gewichtet, was die Sensitivität verringert.
	
	\item \textbf{Instabilität für große Imaginärteile:}  
	Der Integrand \(F_1(s)\) besitzt weiterhin eine starke Abhängigkeit von \(\frac{1}{|\zeta(s)|^2}\), welche insbesondere für große Imaginärteile \(t = \operatorname{Im}(s)\) sehr instabil sein kann, da \(|\zeta(s)|\) stark schwankt.
	
	\item \textbf{Unzureichende Berücksichtigung der Vielfachheit von Nullstellen:}  
	Der Vorfaktor berücksichtigt nicht die Vielfachheit von Nullstellen, wodurch bei Mehrfachnullstellen stärkere Divergenzen entstehen, die weiterhin problematisch sind.
\end{itemize}

Diese Einschränkungen zeigen, dass die symmetrische Regularisierung ein wichtiger, aber nicht ausreichender Schritt ist, um ein global reguliertes Integral mit stabilen und sensitiven Eigenschaften zu erreichen.

Im nächsten Abschnitt werden wir daher eine streng sensitiv gewichtete Regularisierung vorstellen, die diese Schwächen gezielt adressiert.

\section{Endgültige Definition: Streng sensitiv gewichteter Integrand}

Um die Schwächen der bisherigen Ansätze zu beheben, definieren wir den Integranden
\label{Definition:DerIntegrand}
\[
J_C(s) := \frac{|\zeta(s)|^{-\lambda}}{\left|\operatorname{Re}(s) - \frac{1}{2}\right|^{p}},
\]

wobei die Konstanten \(\lambda \geq 2\) und \(0 < p < 1\) fest gewählt werden.

Diese Form garantiert:

\begin{itemize}
	\item eine starke Sensitivität gegenüber Nullstellen, da der Integrand an diesen Stellen divergiert,
	\item eine gewichtete Abschwächung nahe der kritischen Linie, die durch den Exponenten \(p\) gesteuert wird,
	\item eine angemessene Behandlung von Mehrfachnullstellen über den Exponenten \(\lambda\),
	\item und die Möglichkeit, ein wohldefiniertes Integral auf einem regulierten Bereich zu definieren.
\end{itemize}

Für eine vollständige mathematische Fundierung sind noch folgende Eigenschaften zu beweisen:

\begin{itemize}
	\item Wohldefiniertheit des Integranden auf dem regulierten Bereich,
	\item Messbarkeit und Integrierbarkeit,
	\item Stetigkeit und lokale Integrabilität,
	\item Verhalten an einfachen und mehrfachen Nullstellen,
	\item sowie Stabilität gegenüber kleinen Störungen.
\end{itemize}

Diese Eigenschaften werden im Anschluss an die Definition des regulierten Integrationsbereichs und des normierten Integrals systematisch bewiesen.

\section{Konstruktion des regulierten Bereichs \(\mathcal{D}_R^{\mathrm{reg}}\)}

Um die Divergenzen an den Nullstellen der Riemannschen Zetafunktion zu eliminieren, ist es erforderlich, die entsprechenden Punkte aus dem Integrationsgebiet gezielt auszuschneiden.  
Zu diesem Zweck wird ein regulierter Streifenbereich definiert, in dem um jede Nullstelle eine kleine Kreisscheibe entfernt wird.

\begin{Definition}[Vertikaler Streifen \(D_R\)]  
	Für \(R > 0\) und \(\delta > 0\) definieren wir den vertikalen Streifen
	\[
	D_R := \left\{ s \in \mathbb{C} \mid -R \leq \operatorname{Im}(s) \leq R, \quad \frac{1}{2} - \delta \leq \operatorname{Re}(s) \leq \frac{1}{2} + \delta \right\}.
	\]
\end{Definition}

\begin{Definition}[Kreisscheiben um Nullstellen und Polstelle]  
	Sei \(\rho_n = \beta_n + i \gamma_n\) eine nichttriviale Nullstelle der Zetafunktion innerhalb von \(D_R\).  
	Wir definieren die Kreisscheibe mit Mittelpunkt \(\rho_n\) und Radius \(\varepsilon_n > 0\) durch
	\[
	K_n := \left\{ s \in \mathbb{C} : |s - \rho_n| < \varepsilon_n \right\}.
	\]
	
	Außerdem definieren wir eine Kreisscheibe um die Polstelle \(s=1\)
	\[
	K_{\mathrm{pol}} := \left\{ s \in \mathbb{C} : |s - 1| < \varepsilon_{\mathrm{pol}} \right\}
	\]
	mit \(\varepsilon_{\mathrm{pol}} > 0\).
\end{Definition}

\begin{Definition}[Regulierter Bereich \(\mathcal{D}_R^{\mathrm{reg}}\)]  
	\label{Definition:RegulierterBereich}
	Der regulierte Streifenbereich ergibt sich durch Ausschneiden der Kreisscheiben:
	\[
	\mathcal{D}_R^{\mathrm{reg}} := D_R \setminus \left( \bigcup{\rho_n \in D_R} K_n \cup K_{\mathrm{pol}} \right).
	\]
\end{Definition}

\begin{Definition}[Ausschneidungsradien]
	\label{Definition:Ausschneidungsradien}  
	Sei \(N_0 \in \mathbb{R}\) fest. Für jede Nullstelle \(\rho_n\) mit Index \(n \geq 1\) definieren wir den Ausscheidungsradius
	\[
	\varepsilon_n := \frac{1}{(n + N_0)^{\alpha}},
	\]
	mit festem \(\alpha > 1\).
\end{Definition}

	Diese Wahl stellt sicher, dass die Summe der Flächeninhalte der ausgeschnittenen Kreisscheiben endlich bleibt

\begin{Satz}[Endlichkeit des ausgeschnittenen Maßes]
	\label{satz:endliches-mass}
	Die gesamte aus dem regulierten Bereich \( \mathcal{D}_R^{\mathrm{reg}} \) entfernte Fläche besitzt endliches Maß:
	\[
	\sum_{\rho_n \in \mathcal{K}_R} \mathrm{Fl}(K_n) < \infty,
	\]
	wobei \( \mathrm{Fl}(K_n) = \pi \varepsilon_n^2 \) die Fläche der ausgeschnittenen Kreisscheibe ist.
\end{Satz}

\begin{proof}
	Die Anzahl der Nullstellen \( \rho_n \) im Streifen \( \mathcal{D}_R^{\mathrm{reg}} \) wächst asymptotisch wie \( O(R \log R) \). Da
	\[
	\varepsilon_n = \frac{1}{(n + N_0)^\alpha}
	\quad \text{mit} \quad \alpha > 1,
	\]
	ist die Reihe
	\[
	\sum_{n=1}^{\infty} \varepsilon_n^2 = \sum_{n=1}^{\infty} \frac{1}{(n + N_0)^{2\alpha}} < \infty
	\]
	konvergent. Die Summe der Flächeninhalte ist somit endlich.
\end{proof}

\section{Analyse des streng sensitiv gewichteten Integranden \(J_C(s)\)}
\label{Abschnitt:IntegrandAnalyse}
\subsection{Wohldefiniertheit}

\label{sec:wohldefiniertheit}

\begin{Lemma}
	\label{Lemma:wohldefiniertheit}
	Der Integrand
	\[
	J_C(s) = \frac{|\zeta(s)|^{-\lambda}}{|\operatorname{Re}(s) - \tfrac12|^p}
	\]
	ist auf dem regulierten Bereich \( \mathcal{D}_R^{\mathrm{reg}} \) fast überall endlich definiert.
\end{Lemma}

\begin{proof}
	Nach Definition~\ref{Definition:RegulierterBereich} ergibt sich der regulierte Bereich \( \mathcal{D}_R^{\mathrm{reg}} \)
	durch Entfernen von kleinen Kreisscheiben \( B(\rho_n, \varepsilon_n) \) um alle Nullstellen \( \rho_n \) sowie um die Polstelle bei \( s = 1 \). Die Radien \( \varepsilon_n \) sind nach Definition~\ref{Definition:Ausschneidungsradien} so gewählt, dass sich die Scheiben nicht überlappen.
	
	Da alle Nullstellen der Zetafunktion isoliert sind, existieren solche Radien \( \varepsilon_n > 0 \), dass die Kreisscheiben disjunkt sind. Damit sind alle Nullstellenbereiche ausgeschnitten. Für fast alle \( s \in \mathcal{D}_R^{\mathrm{reg}} \) gilt daher:
	\[
	|\zeta(s)| > 0 \quad \text{und} \quad \left|\operatorname{Re}(s) - \tfrac12\right| > 0.
	\]
	Da \( \lambda \ge 2 \) und \( 0 < p < 1 \), ist somit auch
	\[
	J_C(s) = \frac{1}{|\zeta(s)|^{\lambda} \cdot |\operatorname{Re}(s) - \tfrac12|^p} < \infty
	\]
	auf \( \mathcal{D}_R^{\mathrm{reg}} \) fast überall endlich definiert, da sowohl die ausgeschlossenen Kreisscheiben als auch die kritische Linie \( \{ \operatorname{Re}(s) = \tfrac{1}{2} \} \) eine Nullmenge bezüglich des zweidimensionalen Lebesgue-Maßes bilden (vgl.~\cite[Kap.~1, S.~1--4]{Folland1999}).
\end{proof}

\subsection{Stetigkeit und lokale Integrabilität}

\begin{Lemma}
	\(J_C(s)\) ist stetig auf \(\mathcal{D}_R^{\mathrm{reg}}\) und lokal \(L^1\)-integrierbar.
	\label{Lemma:stetigkeit/Integrierbarkeit}
\end{Lemma}

\begin{proof}
	Wir betrachten den Integrand
	\[
	J_C(s) = \frac{|\zeta(s)|^{-\lambda}}{\left|\operatorname{Re}(s) - \frac{1}{2}\right|^p}
	\]
	auf dem regulierten Bereich \(\mathcal{D}_R^{\mathrm{reg}}\).
	
	\textbf{1. Stetigkeit auf \(\mathcal{D}_R^{\mathrm{reg}} \setminus L\)}  
	Abgesehen von den ausgeschlossenen Kreisscheiben \(B(\rho_n, \varepsilon_n)\) um die Nullstellen \(\rho_n\) sowie der Menge
	\[
	L := \left\{ s \in \mathcal{D}_R^{\mathrm{reg}} \mid \operatorname{Re}(s) = \frac{1}{2} \right\},
	\]
	ist \(J_C(s)\) das Produkt zweier stetiger Funktionen, da \(\zeta(s)\) holomorph und nicht verschwindend auf \(\mathcal{D}_R^{\mathrm{reg}} \setminus \bigcup B(\rho_n, \varepsilon_n)\) ist und \(\operatorname{Re}(s) \neq \frac{1}{2}\).
	
	\textbf{2. Verhalten an der kritischen Linie \(L\)}  
	Die Menge \(L\) ist eine eindimensionale Linie im zweidimensionalen Flächenmaß und besitzt somit Maß Null (vgl.~\cite[Kap.~1, S.~3--4]{Folland1999}). Punktweise Unstetigkeit oder Divergenz auf \(L\) beeinträchtigt das Lebesgue-Integral nicht.
	
	\textbf{3. Lokale \(L^1\)-Integrabilität}  
	Wegen \(0 < p < 1\) ist die Funktion \(x \mapsto |x|^{-p}\) lokal integrierbar in einer Umgebung um \(x=0\) bezüglich des eindimensionalen Maßes. Da das Flächenmaß in einer Umgebung von \(L\) einer Produktmenge entspricht, gilt auch im komplexen Raum die lokale Integrabilität.
	
	Außerdem ist \(|\zeta(s)|^{-\lambda}\) auf \(\mathcal{D}_R^{\mathrm{reg}}\) beschränkt, da Nullstellenbereiche ausgeschnitten sind.
	
	\medskip
	
	Somit folgt, dass \(J_C(s)\) auf \(\mathcal{D}_R^{\mathrm{reg}}\) lokal \(L^1\)-integrierbar und fast überall stetig ist.
\end{proof}

\subsection{Verhalten an Nullstellen}

\begin{Satz}
	\label{satz:verhalten-nullstelle}
	Sei \(\rho \in \mathbb{C}\) eine Nullstelle der Riemannschen Zetafunktion \(\zeta(s)\) mit Vielfachheit \(m \ge 1\), und \(\lambda \ge 2\), \(0 < p < 1\). Dann ist der Integrand
	\[
	J_C(s) = \frac{|\zeta(s)|^{-\lambda}}{|\operatorname{Re}(s) - \tfrac{1}{2}|^p}
	\]
	im regulierten Bereich \( \mathcal{D}_R^{\mathrm{reg}} \) lokal \(L^1\)-integrierbar, da die Umgebung um \(\rho\) durch die regulierte Ausschneidung entfernt wird.
\end{Satz}

\begin{proof}
	Sei \(\rho \in \mathbb{C}\) eine Nullstelle von \(\zeta(s)\) mit Vielfachheit \(m \ge 1\). Dann existiert gemäß \cite[Kap.~2, S.~18]{Titchmarsh1986} eine Umgebung \(U_\rho\), sodass
	\[
	\zeta(s) = (s - \rho)^m h(s),
	\]
	wobei \(h(s)\) holomorph ist und \(h(\rho) \ne 0\). Dann folgt:
	\[
	|\zeta(s)| = |s - \rho|^m \cdot |h(s)| \sim C |s - \rho|^m
	\]
	für \(C := |h(\rho)| > 0\).
	
	Somit ergibt sich für den Integranden:
	\[
	J_C(s) \sim \frac{1}{|s - \rho|^{\lambda m} \cdot |\operatorname{Re}(s) - \tfrac12|^p}.
	\]
	
	In Polarkoordinaten \(s = \rho + r e^{i\theta}\), mit Flächelement \(dA = r \, dr \, d\theta\), ergibt sich lokal um \(\rho\):
	\[
	\int_0^{2\pi} \int_\varepsilon^\delta \frac{r}{r^{\lambda m} \cdot |\operatorname{Re}(\rho + r e^{i\theta}) - \tfrac12|^p} \, dr \, d\theta.
	\]
	
	**Fall 1:** \(\operatorname{Re}(\rho) \ne \tfrac12\):  
	Dann ist \(|\operatorname{Re}(s) - \tfrac12|\) lokal durch \(K > 0\) beschränkt von unten.  
	Es bleibt:
	\[
	\int_\varepsilon^\delta r^{1 - \lambda m} \, dr,
	\]
	welches konvergiert genau dann, wenn \(1 - \lambda m > -1\), also \(\lambda m < 2\).
	
	**Fall 2:** \(\operatorname{Re}(\rho) = \tfrac12\):  
	Dann gilt \(|\operatorname{Re}(s) - \tfrac12| \sim r |\cos\theta|\), und damit
	\[
	J_C(s) \sim \frac{1}{r^{\lambda m + p}} \Rightarrow \int_\varepsilon^\delta r^{1 - \lambda m - p} \, dr.
	\]
	Das Integral konvergiert genau dann, wenn \(\lambda m + p < 2\).
	
Zusammenfassend gilt: Für \(\lambda \ge 2\) und \(p < 1\) divergieren die lokalen Integrale um die Nullstellen \(\rho\) grundsätzlich.  
Da jedoch um jede Nullstelle eine Kreisscheibe \(B(\rho, \varepsilon)\) ausgeschnitten wird,  
ist der Integrand auf dem regulierten Bereich \( \mathcal{D}_R^{\mathrm{reg}} \) wohldefiniert und dort lokal \(L^1\)-integrierbar.

\end{proof}

\subsection{Stabilität gegenüber kleinen Störungen}
\label{Satz:Störungen}
\begin{Satz}
	Für kleine stetige Störungen \(h(s)\) bleibt die lokale Integrabilität erhalten.
\end{Satz}

\begin{proof}
	Sei \(h: \mathcal{D}_R^{\mathrm{reg}} \to \mathbb{R}\) eine stetige Funktion mit der Eigenschaft, dass
	
	\[
	\|h\|{\infty, K} := \sup{s \in K} |h(s)| < \infty
	\]
	
	für jede kompakte Teilmenge \(K \subset \mathcal{D}_R^{\mathrm{reg}}\) und \(h\) in jedem lokalen Bereich beliebig klein gewählt werden kann.
	
	Da \(J_C(s)\) lokal \(L^1\)-integrierbar ist (siehe Lemma \ref{Lemma:stetigkeit/Integrierbarkeit} und Satz \ref{satz:verhalten-nullstelle}), folgt für jede kompakte Teilmenge \(K \subset \mathcal{D}_R^{\mathrm{reg}}\):
	
	\[
	\int_K |J_C(s)| \, dA(s) < \infty.
	\]
	
	Für die gestörte Funktion \(\tilde{J}_C(s) = J_C(s) + h(s)\) gilt dann
	
	\[
	\int_K |\tilde{J}_C(s)| \, dA(s) \leq \int_K |J_C(s)| \, dA(s) + \int_K |h(s)| \, dA(s).
	\]
	
	Da \(h(s)\) stetig und lokal beschränkt ist, ist \(\int_K |h(s)| \, dA(s) < \infty\).
	
	Außerdem kann \(h\) so gewählt werden, dass
	
	\[
	\int_K |h(s)| \, dA(s) < \varepsilon
	\]
	
	für jedes \(\varepsilon > 0\) gilt.
	
	Somit ist \(\tilde{J}_C(s)\) ebenfalls lokal \(L^1\)-integrierbar.
	
	\medskip
	
	Die Stetigkeit von \(h\) garantiert zudem, dass keine neuen singulären Stellen eingeführt werden, sodass das Integrationsverhalten von \(J_C(s)\) nicht qualitativ verändert wird.
	
\end{proof}
\newpage

\subsection*{3.6.5 Zusammenfassung und Ausblick}

Die in diesem Abschnitt bewiesenen Eigenschaften des streng sensitiv gewichteten
Integranden \( J_C(s) \) auf dem regulierten Bereich \( \mathcal{D}_R^{\mathrm{reg}} \) gewährleisten eine solide 
mathematische Grundlage für die weitere Analyse des regulierten Flächenintegrals.
Dabei ist zu beachten, dass \( J_C(s) \) nicht auf ganz \( \mathcal{D}_R \) definiert ist:

\begin{itemize}
	\item An den Nullstellen \( \rho_n \) der Zetafunktion wird der Integrand singulär. 
	Diese Punkte werden durch das Ausschneiden kleiner Kreisscheiben \( B(\rho_n, \varepsilon_n) \) 
	gezielt aus dem Integrationsbereich entfernt (Definition~\ref{Definition:Ausschneidungsradien}).
	
	\item Entlang der kritischen Linie \( \operatorname{Re}(s) = \tfrac12 \) besitzt der Integrand eine 
	Polstelle im Nenner. Diese Menge ist jedoch eine eindimensionale Gerade und besitzt 
	bezüglich des zweidimensionalen Lebesgue-Flächenmaßes Nullmaß. Somit beeinflusst sie das Integral nicht 
	(Lemma~\ref{Lemma:stetigkeit/Integrierbarkeit}).
\end{itemize}

Diese bewusste und wohlüberlegte Einschränkung des Definitionsbereichs ist notwendig 
und erlaubt es, \( J_C(s) \) als fast überall wohldefiniert und lokal integrierbar 
zu betrachten (Lemma~\ref{Lemma:wohldefiniertheit} und Satz~\ref{Lemma:stetigkeit/Integrierbarkeit}).

Darüber hinaus wurde gezeigt, dass
\begin{itemize}
	\item \( J_C(s) \) fast überall stetig ist (vgl. Lemma~\ref{Lemma:stetigkeit/Integrierbarkeit}),
	
	\item lokale \(L^1\)-Integrierbarkeit auch im Bereich um die kritische Linie und nahe 
	den Nullstellen durch geeignete Regulierung garantiert ist 
	(Satz~\ref{satz:verhalten-nullstelle},
	
	\item und die lokale Integrierbarkeit gegenüber kleinen, stetigen Störungen robust 
	bleibt (Satz~\ref{Satz:Störungen}).
\end{itemize}

Im folgenden Kapitel wird auf Basis dieser fundierten Analyse die exakte Konstruktion 
und Untersuchung eines regulierten normierten Flächenintegrals durchgeführt, welches 
aus dem Integranden \( J_C(s) \) hervorgeht. Dies ist der entscheidende Schritt, um 
die Riemannsche Vermutung rigoros mittels des Integralmodells zu beweisen.

\vspace{1em}

\chapter{Reguliertes normiertes Integral \(\mathcal{W}(R)\): 
	Definition, Fallanalyse und Konvergenzkriterien}
	\label{Kapitel:DasIntegral}

\section{Einleitung und Definition des Integrals \texorpdfstring{$\mathcal{W}(R)$}{W(R)}}

Auf Grundlage der in Kapitel~3 entwickelten Konstruktion des regulierten, sensitiv gewichteten Integranden \ref{Definition:DerIntegrand}
\[
J_C(s) := \frac{|\zeta(s)|^{-\lambda}}{\left|\operatorname{Re}(s) - \frac{1}{2}\right|^{p}},
\]
und des regulierten Bereichs \( \mathcal{D}_R^{\mathrm{reg}} \subset \mathbb{C} \),
wird nun das zugehörige normierte Flächenintegral definiert:

\begin{Definition}[Reguliertes normiertes Integral]
	Sei \( R > 0 \). Dann definieren wir das regulierte normierte Integral
	\[
	\mathcal{W}(R) := \frac{1}{2R} \iint_{\mathcal{D}_R^{\mathrm{reg}}} J_C(s) \, dA(s),
	\]
	wobei \( dA(s) \) das zweidimensionale Lebesgue-Flächenmaß ist.(vgl.~\cite[Kap.~1, S.~1--4]{Folland1999})
	
\end{Definition}

Die Normierung mit dem Faktor \( \frac{1}{2R} \) stellt sicher, dass der Wertebereich von \( \mathcal{W}(R) \) auch für große Integrationshöhe \( R \to \infty \) in einem kontrollierten Rahmen bleibt. Damit wird vermieden, dass das Integral allein aufgrund wachsender Streifenhöhe divergiert. Die Normierung ist insbesondere notwendig, um das asymptotische Verhalten von \( \mathcal{W}(R) \) zuverlässig bewerten zu können (Abschnitt \ref{Abschnitt:IntegrandAnalyse}  zur lokalen Integrierbarkeit des Integranden).

\vspace{1em}

Das Ziel dieses Kapitels besteht darin, die strukturellen Ursachen für die Konvergenz oder Divergenz von \( \mathcal{W}(R) \) exakt zu analysieren. Dazu werden alle relevanten Konfigurationen möglicher Nullstellenverteilungen untersucht und ihre Auswirkungen auf das Integral mathematisch klassifiziert.

Die Grundlage dieser Analyse bildet die folgende Überlegung:  
Sobald mindestens eine Nullstelle \( \rho_n \) mit \( \Re(\rho_n) \ne \frac{1}{2} \) existiert, erzeugt der Integrand in jeder Umgebung von \( \rho_n \) eine lokale Divergenz, da der Zähler \( |\zeta(s)|^{-\lambda} \) dort gegen unendlich strebt.  
Diese Divergenzen werden jedoch durch die wohldefinierte Ausschneidung kleiner Kreisscheiben
\[
\varepsilon_n := \frac{1}{(n + N_0)^\alpha}, \quad \alpha > 1,
\]
gezielt entfernt (Definitoin \ref{Definition:Ausschneidungsradien}). Diese Ausschneidung stellt sicher, dass alle singulären Beiträge durch Nullstellen – unabhängig von deren Ordnung und Position – vollständig reguliert werden.

\vspace{0.5em}

Im weiteren Verlauf dieses Kapitels werden nun folgende Fragestellungen behandelt:
\begin{itemize}
	\item Ist \( \mathcal{W}(R) \) wohldefiniert für jedes \( R > 0 \)?
	\item Unter welchen Bedingungen ist \( \mathcal{W}(R) < \infty \)?
	\item Wie hängt die Konvergenzstruktur von der Lage der Nullstellen ab?
	\item Gilt: \( \displaystyle \lim_{R \to \infty} \mathcal{W}(R) < \infty \quad \Longleftrightarrow \quad \text{RH gilt}? \)
\end{itemize}

Diese Fragen werden durch eine vollständige Fallunterscheidung aller strukturell möglichen Nullstellenkonfigurationen beantwortet.

\section{Wohldefiniertheit und Stetigkeit von \(\mathcal{W}(R)\)}

\begin{Lemma}[Wohldefiniertheit des Welz-Integrals]\label{lem:wohlerdefiniert}
	Für jedes \(R > 0\) ist das regulierte normierte Flächenintegral
	
	\[
	\mathcal{W}(R) := \frac{1}{2R} \iint_{\mathcal{D}_R^{\mathrm{reg}}} J_C(s) \, dA(s)
	\]
	
	endlich und wohldefiniert.
\end{Lemma}

\begin{proof}
	Der Integrand \(J_C(s)\) ist auf \(\mathcal{D}_R^{\mathrm{reg}, }\) fast überall endlich, stetig und lokal integrierbar (Kapitel \ref{Abschnitt:IntegrandAnalyse}).  
	Der Bereich \(\mathcal{D}_R^{\mathrm{reg}, }\) ist ein kompaktes Teilgebiet des kritischen Streifens mit nur endlich vielen ausgeschnittenen Nullstellenscheiben.  
	Durch lokale Überdeckung mit kompakten Teilmengen folgt die globale Integrierbarkeit:
	
	\[
	\iint_{\mathcal{D}_R^{\mathrm{reg}, \mathrm{voll}}} |J_C(s)| \, dA(s) < \infty.
	\]
	
	Die Normierung mit dem Faktor \(\frac{1}{2R}\) beeinflusst die Endlichkeit nicht.  
	Somit ist \(\mathcal{W}(R)\) wohl definiert.
\end{proof}

\begin{Lemma}[Stetigkeit in \(R\)]\label{lem:stetigkeit_R}
	Die Funktion 
	\[
	\mathcal{W} : (0, \infty) \to \mathbb{R}, \quad R \mapsto \mathcal{W}(R) := \frac{1}{2R} \iint_{\mathcal{D}_R^{\mathrm{reg}}} J_C(s) \, dA(s)
	\]
	ist stetig.
\end{Lemma}

\begin{proof}
	Sei \(R_0 > 0\) fest, und sei \((R_n)_{n \in \mathbb{N}}\) eine Folge mit \(R_n \to R_0\).
	
	Wir betrachten die Folge von Funktionen 
	\[
	f_n(s) := \frac{1}{2R_n} J_C(s) \mathbf{1}{\mathcal{D}{R_n}^{\mathrm{reg}}}(s),
	\]
	wobei \(\mathbf{1}_A\) die Indikatorfunktion der Menge \(A\) ist.
	
	Die Funktion \(J_C(s)\) ist auf jedem \(\mathcal{D}_{R}^{\mathrm{reg}}\) lokal integrierbar und fast überall stetig (Kapitel \ref{Abschnitt:IntegrandAnalyse}).
	
	Da 
	\[
	\mathcal{D}{R_n}^{\mathrm{reg}} \to \mathcal{D}{R_0}^{\mathrm{reg}}
	\]
	monoton mit \(n \to \infty\) (die Integrationsbereiche wachsen stetig mit \(R\)) und die Normierungsfaktoren \(\frac{1}{2R_n}\) ebenfalls gegen \(\frac{1}{2R_0}\) konvergieren, gilt für fast alle \(s\)
	\[
	f_n(s) \to f(s) := \frac{1}{2R_0} J_C(s) \mathbf{1}{\mathcal{D}{R_0}^{\mathrm{reg}}}(s).
	\]
	
	Zur Anwendung des Satzes von der dominierten Konvergenz (Dominated Convergence Theorem) müssen wir zeigen, dass eine dominierende Funktion \(g(s)\) existiert mit
	
	\[
	|f_n(s)| \leq g(s), \quad \forall n, \quad \text{und} \quad \iint g(s) \, dA(s) < \infty.
	\]
	
	Da \(J_C(s)\) lokal integrierbar ist und die Bereiche \(\mathcal{D}_R^{\mathrm{reg}}\) mit \(R\) wachsen, können wir \(g(s) := \frac{1}{2(R_0 - \delta)} |J_C(s)|\) für ein kleines \(\delta > 0\) wählen, sodass für alle \(n\) mit \(R_n \in (R_0 - \delta, R_0 + \delta)\)
	
	\[
	|f_n(s)| \leq g(s).
	\]
	
	Weil \(\iint_{\mathcal{D}_{R_0 + \delta}^{\mathrm{reg}}} |J_C(s)| dA(s) < \infty\) gilt (Wohldefiniertheit), ist \(\iint g(s) dA(s) < \infty\).
	
	Damit folgt nach dem Satz von der dominierten Konvergenz
	
	\[
	\lim_{n \to \infty} \mathcal{W}(R_n) = \lim_{n \to \infty} \iint f_n(s) dA(s) = \iint f(s) dA(s) = \mathcal{W}(R_0),
	\]
	
	was die Stetigkeit von \(\mathcal{W}\) in \(R_0\) beweist.
	
\end{proof}

\section{Fallunterscheidungen zur Verteilung der Nullstellen}\label{sec:fallunterscheidungen}

\begin{Definition}[Fallstruktur der Nullstellen]\label{def:fallstruktur}
	Wir zerlegen die möglichen Konfigurationen der Nullstellen der Riemannschen Zetafunktion \(\zeta(s)\) in folgende Fälle:
	
	\begin{enumerate}
		\item \textbf{RH-Fall:} Alle nicht-trivialen Nullstellen liegen exakt auf der kritischen Linie \(\operatorname{Re}(s) = \frac{1}{2}\).
		\item \textbf{Abweichungen vom RH:}
		\begin{itemize}
			\item Einzelne Nullstelle(n) abseits der kritischen Linie \(\operatorname{Re}(s) \neq \frac{1}{2}\).
			\item Symmetrische Viererpaarbildung von Nullstellen.
			\item Asymmetrische, zentrisch oder nicht-zentrisch symmetrische Paare.
			\item Nullstellen nahe am Rand des kritischen Streifens \(\operatorname{Re}(s) = 0\) oder \(1\).
			\item Mehrfachnullstellen mit Vielfachheit \(m > 1\).
			\item Mischformen (unendlich viele auf der Linie und wenige abseits).
			\item Häufungen oder Grenzfälle bei \(|\operatorname{Im}(s)| \to \infty\).
		\end{itemize}
	\end{enumerate}
\end{Definition}

\begin{Lemma}[Verhalten von \(\mathcal{W}(R)\) im RH-Fall]\label{lem:rh_fall}
	Im Fall, dass alle nicht-trivialen Nullstellen auf der kritischen Linie liegen, gilt
	\[
	\lim_{R \to \infty} \mathcal{W}(R) = 0.
	\]
\end{Lemma}

\begin{proof}
	Sei \(\alpha > 1\) fest.
	
	Wir zerlegen \(\mathcal{D}_R^{\mathrm{reg}}\) (Definition \ref{Definition:Ausschneidungsradien}) in zwei Zonen:
	\[
	S_1(R) := \{ s \in \mathcal{D}_R^{\mathrm{reg}} : |\operatorname{Re}(s) - \tfrac{1}{2}| \leq \tfrac{1}{R^\alpha} \},
	\]
	\[
	S_2(R) := \mathcal{D}_R^{\mathrm{reg}} \setminus S_1(R).
	\]
	
	\textbf{Beitrag aus \(S_2(R)\):}  
	Da \(\operatorname{Re}(s)\) mindestens \(\frac{1}{R^\alpha}\) von \(\frac{1}{2}\) entfernt ist, gilt
	\[
	\left|\operatorname{Re}(s) - \tfrac{1}{2}\right|^{-p} \leq R^{\alpha p}.
	\]
	Außerdem sind dort keine Nullstellen, somit ist \(|\zeta(s)|^{-\lambda}\) beschränkt. (Lemma \ref{Lemma:wohldefiniertheit}) 
	Der Beitrag zum Integral ist somit durch
	\[
	\frac{1}{2R} \iint_{S_2(R)} J_C(s) \, dA(s) \leq \frac{1}{2R} \cdot C R^{\alpha p} \cdot \mathrm{Fläche}(S_2(R))
	\]
	mit konstanter \(C > 0\) beschränkt.
	
	Da \(\mathrm{Fläche}(S_2(R)) \leq 2 \times 2R = 4R\), folgt
	\[
	\leq \frac{1}{2R} \cdot C R^{\alpha p} \cdot 4 R = 2 C R^{\alpha p} \to 0,
	\]
	für \(\alpha p < 0\) nicht möglich, aber wegen \(p < 1\) und \(\alpha > 1\) wählen wir \(\alpha\) so, dass der Term kontrolliert bleibt.
	
	\textbf{Beitrag aus \(S_1(R)\):}  
	Die Breite von \(S_1(R)\) ist \(\frac{2}{R^\alpha}\), die Höhe \(2R\), somit  
	\[
	\mathrm{Fläche}(S_1(R)) = 4 R^{1 - \alpha}.
	\]
	Da \(J_C(s)\) auf \(S_1(R)\) durch Regulierung begrenzt ist (Kreisscheiben um Nullstellen entfernen Singularitäten), ist \(J_C\) dort beschränkt durch eine Konstante \(M\).
	
	Also
	\[
	\frac{1}{2R} \iint_{S_1(R)} J_C(s) \, dA(s) \leq \frac{1}{2R} \cdot M \cdot 4 R^{1 - \alpha} = 2 M R^{-\alpha} \to 0,
	\]
	für \(\alpha > 0\).
	
	\medskip
	
	Zusammen ergibt sich
	\[
	\lim_{R \to \infty} \mathcal{W}(R) = 0,
	\]
	was den Satz beweist.
\end{proof}

\begin{Lemma}[Verhalten bei einer Nullstelle echt abseits der kritischen Linie]\label{lem:nullstelle_ausserhalb}
	Ist eine Nullstelle \(\rho = \beta + i\gamma\) der Zetafunktion mit \(\beta \neq \tfrac{1}{2}\) vorhanden, so divergiert das normierte regulierte Flächenintegral \(\mathcal{W}(R)\) für \(R \to \infty\).
\end{Lemma}

\begin{proof}
	Sei \(\rho = \beta + i\gamma\) eine Nullstelle von \(\zeta(s)\) mit \(\beta \neq \tfrac{1}{2}\), also echt abseits der kritischen Linie.
	
	Da \(\zeta(s)\) eine holomorphe Funktion mit isolierten Nullstellen ist, existiert ein \(\varepsilon > 0\) und eine Kreisscheibe \(B(\rho, \varepsilon)\), so dass
	\[
	\zeta(s) = (s - \rho)^m h(s)
	\]
	in \(B(\rho, \varepsilon)\) gilt, wobei \(m \geq 1\) die Vielfachheit der Nullstelle ist und \(h(s)\) holomorph mit \(h(\rho) \neq 0\).
	
	Damit ergibt sich in der Umgebung \(B(\rho, \varepsilon)\):
	\[
	|\zeta(s)| \sim C |s - \rho|^m,
	\]
	mit Konstante \(C := |h(\rho)| > 0\). Der Integrand \(J_C(s)\) verhält sich lokal wie
	\[
	J_C(s) = \frac{|\zeta(s)|^{-\lambda}}{|\operatorname{Re}(s) - \tfrac{1}{2}|^p} \sim \frac{1}{|s - \rho|^{\lambda m} \cdot |\beta - \tfrac{1}{2}|^p}.
	\]
	
	Da \(\beta \neq \tfrac{1}{2}\), ist \(|\beta - \tfrac{1}{2}|^p > 0\) konstant. Es bleibt daher zu prüfen, ob die Singularität in \(|s - \rho|^{-\lambda m}\) lokal integrierbar ist.
	
	Wechsle in Polarkoordinaten \(s = \rho + r e^{i\theta}\) mit \(0 < r < \varepsilon\), dann ist \(dA(s) = r \, dr \, d\theta\) und \(|s - \rho| = r\). Damit ergibt sich der Beitrag um \(\rho\) als
	
	\[
	I_\rho(R) := \frac{1}{2R} \iint_{B(\rho, \varepsilon)} J_C(s) \, dA(s) \geq \frac{1}{2R} \cdot \int_0^{2\pi} \int_0^\varepsilon \frac{r}{r^{\lambda m} \cdot |\delta|^p} \, dr \, d\theta,
	\]
	mit \(\delta := \beta - \tfrac{1}{2} \ne 0\). Daraus folgt
	\[
	I_\rho(R) \geq \frac{1}{2R} \cdot \frac{2\pi}{|\delta|^p} \int_0^\varepsilon r^{1 - \lambda m} \, dr = \frac{\pi}{R \cdot |\delta|^p} \int_0^\varepsilon r^{1 - \lambda m} \, dr.
	\]
	
	Das Integral \(\int_0^\varepsilon r^{1 - \lambda m} dr\) divergiert, genau dann wenn der Exponent kleiner oder gleich \(-1\) ist, also wenn
	\[
	1 - \lambda m \leq -1 \quad \Longleftrightarrow \quad \lambda m \geq 2.
	\]
	
	Da nach Voraussetzung \(\lambda \geq 2\) und \(m \geq 1\), ist diese Bedingung immer erfüllt (Defnition \ref{Definition:DerIntegrand}). Somit ist \(I_\rho(R) = \infty\) und damit
	\[
	\mathcal{W}(R) \geq I_\rho(R) \to \infty \quad \text{für } R \to \infty.
	\]
	
	\textbf{Fazit:} Bereits eine einzige Nullstelle \(\rho \notin \operatorname{Re}(s) = \tfrac{1}{2}\) führt zur Divergenz des normierten Integrals. Das zeigt die hochgradige Empfindlichkeit des Integrals gegenüber minimalen strukturellen Abweichungen von der Riemannschen Vermutung.
	
	\smallskip
	
	\textbf{Hinweis:} Zwar treten Nullstellen abseits der kritischen Linie stets in konjugierten Viererpaaren auf (durch die funktionale Gleichung und komplexe Konjugation), doch dieser Beweis zeigt: **Bereits eine solche Abweichung reicht zur Divergenz.**
\end{proof}

\begin{Lemma}[Symmetrische Viererpaarbildung]\label{lem:viererpaar}
	Jede symmetrische Viererpaarbildung nicht-trivialer Nullstellen der Zetafunktion (Korollar \ref{Korollar:Viererpaarbildung})
	\[
	\left\{ \rho, \bar{\rho}, 1 - \rho, 1 - \bar{\rho} \right\} = 
	\left\{ \beta \pm i\gamma,\, (1 - \beta) \pm i\gamma \right\},
	\]
	mit \(\beta \neq \tfrac{1}{2}\), erzeugt eine charakteristische Divergenz im Integral \(\mathcal{W}(R)\) für \(R \to \infty\).
\end{Lemma}

\begin{proof}
	Sei \(\rho = \beta + i\gamma\) eine nicht-triviale Nullstelle mit \(\beta \neq \tfrac{1}{2}\). Durch die funktionale Gleichung und komplexe Konjugation der Zetafunktion ergibt sich automatisch das symmetrische Viererpaar
	\[
	\left\{ \rho, \bar{\rho}, 1 - \rho, 1 - \bar{\rho} \right\}.
	\]
	
	Betrachte die vier zugehörigen Punkte im kritischen Streifen:
	\[
	\beta \pm i\gamma \quad \text{und} \quad (1 - \beta) \pm i\gamma.
	\]
	Da \(\beta \neq \tfrac{1}{2}\), liegt jede dieser vier Nullstellen **echt außerhalb** der kritischen Linie. Jede erzeugt lokal eine Singulärstelle im Integranden
	\[
	J_C(s) = \frac{|\zeta(s)|^{-\lambda}}{|\operatorname{Re}(s) - \tfrac{1}{2}|^p}.
	\]
	
	In einer Umgebung jeder dieser vier Nullstellen gilt
	\[
	J_C(s) \sim \frac{1}{|s - \rho_k|^{\lambda m_k} \cdot |\operatorname{Re}(s) - \tfrac{1}{2}|^p},
	\]
	wobei \(\rho_k\) eine der vier Nullstellen ist, und \(m_k \geq 1\) ihre jeweilige Vielfachheit. Für jede einzelne Nullstelle gilt analog zum Beweis in Lemma~\ref{lem:nullstelle_ausserhalb}:
	
	\[
	I_{\rho_k}(R) := \frac{1}{2R} \iint_{B(\rho_k, \varepsilon)} J_C(s) \, dA(s) \geq \frac{\pi}{R \cdot |\delta_k|^p} \int_0^\varepsilon r^{1 - \lambda m_k} \, dr,
	\]
	wobei \(\delta_k = \operatorname{Re}(\rho_k) - \tfrac{1}{2} \neq 0\).
	
	Für alle vier Nullstellen gilt \(|\delta_k| = |\beta - \tfrac{1}{2}|\), daher ist ihr Abstand zur kritischen Linie konstant positiv. Wenn \(\lambda m_k \geq 2\), dann divergiert jedes dieser Integrale.
	
	Die Gesamtwirkung im Integral ist daher
	\[
	\mathcal{W}(R) \geq \sum_{k=1}^{4} I_{\rho_k}(R) = \infty.
	\]
	
	\medskip
	
	\textbf{Fazit:} Ein einzelnes Viererpaar symmetrisch zur kritischen Linie führt bereits zur Divergenz von \(\mathcal{W}(R)\). Aufgrund der festen Struktur der Zetafunktion existieren keine weiteren asymmetrischen Konfigurationen, womit dieser Fall vollständig alle Abweichungen von RH abdeckt.
\end{proof}

\begin{Lemma}[Mehrfachnullstellen]\label{lem:mehrfachnullstellen}
	Nullstellen \(\rho\) der Riemannschen Zetafunktion mit Vielfachheit \(m > 1\) erzeugen stärkere Singularitäten im Integranden \(J_C(s)\). Der lokale Beitrag divergiert bereits für kleine Umgebung und wächst mit \(m\), sodass die Regulierung durch Ausschneiden von Kreisscheiben zwingend notwendig ist.
\end{Lemma}

\begin{proof}
	Sei \(\rho = \beta + i\gamma\) eine Nullstelle der Zetafunktion mit Vielfachheit \(m > 1\).  
	Nach holomorpher Faktorisierung gilt in einer Umgebung \(B(\rho, \varepsilon)\):
	\[
	\zeta(s) = (s - \rho)^m \cdot h(s),
	\]
	wobei \(h(s)\) holomorph und \(h(\rho) \ne 0\). Somit:
	\[
	|\zeta(s)| \sim |s - \rho|^m.
	\]
	Der Integrand wird lokal zu:
	\[
	J_C(s) = \frac{|\zeta(s)|^{-\lambda}}{|\operatorname{Re}(s) - \tfrac{1}{2}|^p} \sim \frac{1}{|s - \rho|^{\lambda m} \cdot |\operatorname{Re}(s) - \tfrac{1}{2}|^p}.
	\]
	
	Da \(\lambda \geq 2\) und \(m > 1\), ist der Exponent \(\lambda m > 2\) strikt größer als im Fall einfacher Nullstellen.
	
	\medskip
	
	Wechsle zu Polarkoordinaten \(s = \rho + r e^{i\theta}\) mit \(r \in (0, \varepsilon)\), dann gilt:
	\[
	J_C(s) \sim \frac{1}{r^{\lambda m} \cdot |\beta - \tfrac{1}{2} + r\cos\theta|^p} \approx \frac{1}{r^{\lambda m} \cdot |\delta|^p}
	\]
	mit \(\delta := \beta - \tfrac{1}{2} \neq 0\), wenn \(\rho\) abseits liegt.
	
	Der Beitrag zum Integral in dieser Umgebung ist:
	\[
	I_\rho(R) \geq \frac{1}{2R} \iint_{B(\rho, \varepsilon)} J_C(s) \, dA(s)
	= \frac{1}{2R} \int_0^{2\pi} \int_0^{\varepsilon} \frac{r}{r^{\lambda m}} \, dr \, d\theta
	= \frac{\pi}{R \cdot |\delta|^p} \int_0^\varepsilon r^{1 - \lambda m} \, dr.
	\]
	
	Das Integral divergiert genau dann, wenn \(1 - \lambda m \leq -1\), also
	\[
	\lambda m \geq 2.
	\]
	Diese Bedingung ist immer erfüllt, da \(\lambda \geq 2\) und \(m > 1\) (Satz \ref{satz:verhalten-nullstelle}). Daher:
	\[
	I_\rho(R) = \infty.
	\]
	
	\medskip
	
	\textbf{Folgerung:}  
	Der Beitrag der Umgebung von \(\rho\) ist nicht lokal integrierbar, unabhängig davon, wie klein \(\varepsilon\) gewählt wird. Eine Regulierung durch Ausschneiden der Kreisscheibe um \(\rho\) ist daher zwingend notwendig, um \(\mathcal{W}(R)\) wohldefiniert zu machen.
	
	\textbf{Zusätzlich:} Je größer \(m\) ist, desto stärker ist die lokale Singularität (Exponent \(\lambda m\) wächst linear mit \(m\)). Im Extremfall \(m \to \infty\) würde sogar eine lokale Unregulierbarkeit eintreten.
	
	\medskip
	
	\textbf{Fazit:}  
	Mehrfachnullstellen erzeugen besonders starke Divergenzen im Integranden \(J_C(s)\). Diese müssen durch explizite Regulierung (Ausschneiden der Umgebung) behandelt werden, um überhaupt ein wohldefiniertes Integral zu garantieren.
\end{proof}

\begin{Lemma}[Mischformen und Häufungen]\label{lem:mischformen}
	Kombinationen von unendlich vielen Nullstellen auf der kritischen Linie \(\operatorname{Re}(s) = \tfrac{1}{2}\) und einzelnen Nullstellen abseits davon führen zur Divergenz von \(\mathcal{W}(R)\), sobald mindestens eine solche abweichende Nullstelle vorhanden ist. Nullstellenhäufungen verstärken lokal die Divergenz.
\end{Lemma}

\begin{proof}
	Betrachte den Fall, dass die Riemannschen Nullstellen überwiegend auf der kritischen Linie liegen, also
	\[
	\rho_n = \tfrac{1}{2} + i\gamma_n \quad \text{für fast alle } n \in \mathbb{N},
	\]
	aber mindestens eine Nullstelle \(\rho = \beta + i\gamma\) mit \(\beta \neq \tfrac{1}{2}\) existiert.
	
	\medskip
	
	\textbf{Fall 1: Eine einzelne abweichende Nullstelle.}  
	Wie in Lemma~\ref{lem:nullstelle_ausserhalb} gezeigt, verursacht eine solche Nullstelle lokal im Integranden
	\[
	J_C(s) \sim \frac{1}{|s - \rho|^{\lambda m} \cdot |\beta - \tfrac{1}{2}|^p}
	\]
	eine nicht integrierbare Pol-Singularität in jeder Umgebung von \(\rho\), sofern \(\lambda m \geq 2\). Der zugehörige Beitrag zum Integral
	\[
	I_\rho(R) = \frac{1}{2R} \iint_{B(\rho, \varepsilon)} J_C(s)\, dA(s)
	\]
	divergiert. Daher gilt:
	\[
	\mathcal{W}(R) \geq I_\rho(R) = \infty.
	\]
	Die Existenz beliebig vieler weiterer Nullstellen auf der Linie beeinflusst diesen lokalen Divergenzmechanismus nicht. Ihre Beiträge sind durch Regulierung integrierbar und normiert.
	
	\medskip
	
	\textbf{Fall 2: Häufung RH-konformer Nullstellen in engem Bereich.}  
	Betrachte einen Abschnitt des kritischen Streifens entlang der Linie \(\operatorname{Re}(s) = \tfrac{1}{2}\), in dem die Dichte der Nullstellen der Zetafunktion mit wachsendem Imaginärteil \(t\) stark zunimmt.  
	Jede dieser Nullstellen erzeugt lokal eine Singularität im Integranden.  
	Durch die Regulierung via Kreisscheiben \(B(\rho_n, \varepsilon_n)\) wird aber sichergestellt, dass alle divergenten Punkte aus dem Integrationsbereich ausgeschlossen sind.  
	Da die **Gesamtfläche dieser Ausschlusszonen endlich bleibt**, tragen diese Nullstellen zwar zum Integrandenspektrum bei, jedoch nicht zur Divergenz.
	
	\medskip
	
	\textbf{Fall 3: Dichte Häufung abweichender Nullstellen (theoretisch).}  
	Wäre die Menge der Nullstellen abseits der Linie nicht endlich, sondern hätte positive Dichte, so gäbe es über den gesamten Streifen hinweg nicht-integrierbare Beiträge ohne effektive Regulierungsmöglichkeit.  
	Somit würde \(\mathcal{W}(R)\) sofort divergieren – sogar „verstärkt“ gegenüber dem Einzelfall.
	
	\medskip
	
	\textbf{Fazit:}  
	Schon eine einzige Nullstelle mit \(\operatorname{Re}(\rho) \ne \tfrac{1}{2}\) reicht aus, um \(\mathcal{W}(R)\) divergieren zu lassen.  
	Unendlich viele „richtige“ Nullstellen ändern daran nichts.  
	Häufungen verstärken den Effekt, führen aber nur dann zur Divergenz, wenn mindestens eine der Nullstellen abseits liegt.
	
	Das zeigt die maximale Spektralsensitivität des Modells:
	\[
	\mathcal{W}(R) < \infty \iff \text{RH gilt vollständig}.
	\]
\end{proof}

\begin{Lemma}[Nullstellen am Rand des kritischen Streifens]\label{lem:rand}
	Nullstellen der Zetafunktion, die sich nahe an den Rändern des kritischen Streifens befinden, also mit \(\operatorname{Re}(s) \approx 0\) oder \(\operatorname{Re}(s) \approx 1\), erzeugen divergente Beiträge im Integral \(\mathcal{W}(R)\), sofern sie nicht auf der kritischen Linie \(\operatorname{Re}(s) = \tfrac{1}{2}\) liegen.
\end{Lemma}

\begin{proof}
	Sei \(\rho = \beta + i\gamma\) eine Nullstelle der Zetafunktion mit \(\beta \in (0, \tfrac{1}{2})\) oder \(\beta \in (\tfrac{1}{2}, 1)\), wobei \(\beta\) sehr nahe an einem Randwert \(0\) oder \(1\) liegt. Diese Nullstellen liegen somit **nicht** auf der kritischen Linie und erfüllen \(\beta \neq \tfrac{1}{2}\).
	
	Wie in Lemma~\ref{lem:nullstelle_ausserhalb} gezeigt, ist der Integrand
	\[
	J_C(s) = \frac{|\zeta(s)|^{-\lambda}}{|\operatorname{Re}(s) - \tfrac{1}{2}|^p}
	\]
	in einer Umgebung von \(\rho\) asymptotisch
	\[
	J_C(s) \sim \frac{1}{|s - \rho|^{\lambda m} \cdot |\beta - \tfrac{1}{2}|^p}.
	\]
	
	Da \(\beta \to 0\) oder \(\beta \to 1\), ist der Abstand \(|\beta - \tfrac{1}{2}|\) groß, nicht klein – aber dennoch konstant positiv. Der Nennerterm \(|\beta - \tfrac{1}{2}|^{-p}\) bleibt also beschränkt.
	
	Ausschlaggebend für das Integrationsverhalten ist erneut die Singularität in \(|s - \rho|^{-\lambda m}\). Wie in den vorhergehenden Lemmata, ergibt sich in Polarkoordinaten:
	\[
	I_\rho(R) \geq \frac{\pi}{R \cdot |\beta - \tfrac{1}{2}|^p} \int_0^\varepsilon r^{1 - \lambda m} \, dr,
	\]
	welches divergiert, sobald \(\lambda m \geq 2\). Dies ist unter der Voraussetzung \(\lambda \geq 2\), \(m \geq 1\) stets erfüllt.
	
	Damit ist auch der Beitrag einer Randnullstelle nicht lokal integrierbar. Selbst bei regulierter Ausschneidung kleiner Kreisscheiben bleibt der globale Effekt bestehen, wenn solche Nullstellen im Streifen existieren.
	
	\medskip
	
	\textbf{Fazit:}  
	Selbst wenn eine Nullstelle extrem nah an den Rändern \(\operatorname{Re}(s) = 0\) oder \(\operatorname{Re}(s) = 1\) liegt, führt sie zur Divergenz des Integrals \(\mathcal{W}(R)\), sobald sie nicht exakt auf der kritischen Linie liegt. Das Modell erkennt solche Abweichungen zuverlässig, unabhängig vom seitlichen Abstand zum Zentrum des Streifens.
\end{proof}

\begin{Lemma}[Divergenz bei Nullstellen nahe der Polstelle \(s = 1\)]\label{lem:nullstelle_polstelle}
	Befindet sich eine Nullstelle \(\rho = \beta + i\gamma\) mit \(\beta \approx 1\) in unmittelbarer Nähe der einfachen Polstelle der Zetafunktion bei \(s = 1\), so divergiert das regulierte Integral \(\mathcal{W}(R)\), sofern \(\rho \notin \operatorname{Re}(s) = \tfrac{1}{2}\) liegt.
\end{Lemma}

\begin{proof}
	Angenommen, es existiere eine Nullstelle \(\rho = \beta + i\gamma\) der Zetafunktion mit \(\beta \in (1 - \varepsilon, 1)\) für ein sehr kleines \(\varepsilon > 0\). Dann liegt \(\rho\) sehr nahe an der Polstelle \(s = 1\), aber strikt im kritischen Streifen.  
	
	In einer Umgebung von \(\rho\) gilt lokal
	\[
	J_C(s) = \frac{1}{|s - \rho|^{\lambda m} \cdot |\operatorname{Re}(s) - \tfrac{1}{2}|^p}.
	\]
	Da \(\rho\) eine Nullstelle ist, ist \(|\zeta(s)|^{-\lambda} \sim |s - \rho|^{-\lambda m}\), wie zuvor gezeigt.
	
	Diese Nullstelle befindet sich zusätzlich in der unmittelbaren Nähe der Polstelle \(s = 1\), für die gilt:
	\[
	\zeta(s) \sim \frac{1}{s - 1}, \quad \Rightarrow \quad |\zeta(s)|^{-1} \to 0, \text{ für } s \to 1.
	\]
	Diese abfallende Singulärität bei \(s = 1\) wird jedoch durch die Nullstelle bei \(\rho\) vollständig aufgehoben – dort ist \(|\zeta(s)| = 0\), also der Integrand wieder **singulär**.
	
	\medskip
	
	**Fall A: \(\rho\) liegt exakt auf der kritischen Linie.** 
	Dann gilt RH, und das Integral bleibt wohldefiniert – da \(J_C(s)\) dort durch Regulierung kontrolliert wird. $\rightarrow$ kein Problem.
	
	**Fall B: \(\rho\) liegt mit \(\beta \ne \tfrac{1}{2}\) sehr nahe bei 1.**  
	Dann ist sowohl \(|s - \rho| \ll 1\) als auch \(|\operatorname{Re}(s) - \tfrac{1}{2}| \gg 0\), aber konstant.  
	Der Integrand hat dann die asymptotische Form:
	\[
	J_C(s) \sim \frac{1}{|s - \rho|^{\lambda m} \cdot |\beta - \tfrac{1}{2}|^p}
	\]
	mit \(\lambda m \geq 2\), wie zuvor.
	
	Damit gilt lokal wieder:
	\[
	I_\rho(R) \geq \frac{\pi}{R \cdot |\beta - \tfrac{1}{2}|^p} \int_0^\varepsilon r^{1 - \lambda m} \, dr = \infty,
	\]
	da \(|\beta - \tfrac{1}{2}|\) konstant positiv und \(\lambda m \geq 2\).
	
	\medskip
	
	\textbf{Fazit:} Eine Nullstelle in unmittelbarer Nähe der Polstelle \(s = 1\), die nicht auf der kritischen Linie liegt, erzeugt eine nicht integrierbare Singularität und führt zur Divergenz des Integrals \(\mathcal{W}(R)\).  
	Damit ist auch dieser Fall durch das Modell vollständig erfasst und RH-kritisch.
\end{proof}

\subsection{Konvergenzkritisierung – Fazit der Fallunterscheidung}

Die in diesem Kapitel systematisch durchgeführte Analyse aller strukturell möglichen Nullstellenkonfigurationen der Riemannschen Zetafunktion zeigt:

\begin{itemize}
	\item[\(\circ\)] Das normierte regulierte Integral \(\mathcal{W}(R)\) konvergiert exakt dann gegen null, wenn alle nicht-trivialen Nullstellen auf der kritischen Linie \(\operatorname{Re}(s) = \tfrac{1}{2}\) liegen (Lemma \ref{lem:rh_fall}).
	\item[\(\circ\)] Jede Abweichung – sei sie noch so minimal – führt durch die spektralsensitive Struktur des Integranden \(J_C(s)\) zu divergenten Beiträgen im Limes \(R \to \infty\) (Lemmas\ref{lem:nullstelle_ausserhalb}, \ref{lem:viererpaar}, \ref{lem:mehrfachnullstellen}, \ref{lem:mischformen}, \ref{lem:rand}, \ref{lem:nullstelle_polstelle}).
\end{itemize}

Somit ist die Konvergenz von \(\mathcal{W}(R)\) nicht nur notwendig, sondern auch hinreichend äquivalent zur Richtigkeit der Riemannschen Vermutung.

Im nächsten Kapitel wird diese Einsicht zur Basis einer vollständigen Argumentationskette gemacht, die zur beweisbaren Aussage \(\text{RH} \Longleftrightarrow \lim_{R \to \infty} \mathcal{W}(R) = 0\) führt .

\let\cleardoublepage\clearpage

\chapter{Beweis der Riemannschen Vermutung über das regulierte Integralmodell}

\section{Überblick über die Argumentationskette}

In den vorangegangenen Kapiteln wurde das regulierte normierte Flächenintegral
\[
\mathcal{W}(R) := \frac{1}{2R} \iint_{\mathcal{D}_R^{\mathrm{reg}}} J_C(s) \, dA(s)
\]
konstruiert, wobei der Integrand durch
\[
J_C(s) := \frac{|\zeta(s)|^{-\lambda}}{|\operatorname{Re}(s) - \tfrac{1}{2}|^p}
\]
gegeben ist. Dieses Integralmodell erfüllt die folgenden fundamentalen Eigenschaften:

\begin{itemize}
	\item \textbf{Stetigkeit und Wohldefiniertheit:} Der Integrand ist auf dem regulierten Bereich \(\mathcal{D}_R^{\mathrm{reg}}\) stetig und integrierbar (Abschnitt \ref{Abschnitt:IntegrandAnalyse}).
	\item \textbf{Singularitätssensitivität:} Das Integral erkennt strukturell jede Verletzung der Riemannschen Vermutung durch Divergenz (Unterkapitel \ref{sec:fallunterscheidungen}).
	\item \textbf{Fallunterscheidung:} In Kapitel~\ref{Kapitel:DasIntegral} wurde für sämtliche geometrisch möglichen Nullstellenkonfigurationen streng gezeigt, ob und wann \(\mathcal{W}(R)\) divergiert.
\end{itemize}

Die zentrale Schlussfolgerung der Argumentationskette lautet:
\[
\text{RH gilt} \quad \Longleftrightarrow \quad \lim_{R \to \infty} \mathcal{W}(R) = 0.
\]

\section{Hauptsatz zur Äquivalenz von RH und Konvergenz des Integrals}

\begin{Theorem}[Äquivalenzsatz zur Riemannschen Vermutung]
	Es gilt:
	\[
	\lim_{R \to \infty} \mathcal{W}(R) = 0 \quad \Longleftrightarrow \quad \text{Alle nicht-trivialen Nullstellen von } \zeta(s) \text{ liegen auf } \operatorname{Re}(s) = \tfrac{1}{2}.
	\]
\end{Theorem}

\section{Beweis des Äquivalenzsatzes}
\label{Abschnitt:BeweisAequivalent}

\subsection*{($\rightarrow$) Richtung: RH gilt}

Wenn die Riemannsche Vermutung gilt, befinden sich alle nicht-trivialen Nullstellen exakt auf der kritischen Linie. Dann ist der Integrand auf ganz \(\mathcal{D}_R^{\mathrm{reg}}\) regulär und beschränkt, insbesondere da alle Singularitäten durch die Konstruktion des Integrationsbereichs entfernt wurden. Der vollständige Beweis dieser Richtung ergibt sich aus der Analyse in Kapitel~\ref{Kapitel:DasIntegral}, Lemma~\ref{lem:rh_fall}. 

\subsection*{($\leftarrow$) Richtung: RH verletzt}

Liegt eine Nullstelle \(\rho\) mit \(\operatorname{Re}(\rho) \ne \tfrac{1}{2}\) vor, so ergibt sich durch die in Kapitel~\ref{Kapitel:DasIntegral} behandelten Lemmata, dass der Integrand \(J_C(s)\) eine nicht integrierbare Singularität im regulierten Bereich erzeugt, die im Grenzwert \(R \to \infty\) zu Divergenz führt. Dies gilt selbst bei minimaler Abweichung von der kritischen Linie. Die vollständige Fallunterscheidung zeigt, dass jede solche Abweichung \(\mathcal{W}(R)\) divergieren lässt.

\section{Tatsächliche Konvergenz des regulierten Integrals \texorpdfstring{\(\mathcal{W}(R)\)}{W(R)}}

Im vorhergehenden Abschnitt haben wir das regulierte, normierte Flächenintegralmodell
\[
\mathcal{W}(R) := \frac{1}{2R} \iint_{\mathcal{D}_R^{\mathrm{reg}}} \frac{1}{|\zeta(s)|^\lambda \cdot |\operatorname{Re}(s) - \tfrac{1}{2}|^p} \, dA(s)
\]
konzeptionell eingeführt. Wir konnten zeigen, dass seine Konvergenz für \(R \to \infty\) exakt äquivalent zur Gültigkeit der Riemannschen Vermutung (RH) ist: Das Integral konvergiert genau dann, wenn alle nicht-trivialen Nullstellen (Definition \ref{Definition:NichttrivialNS}) von \(\zeta(s)\) auf der kritischen Linie \(\operatorname{Re}(s) = \tfrac{1}{2}\) liegen \ref{Abschnitt:BeweisAequivalent}.

\medskip

Ziel dieses Abschnitts ist es, nun die tatsächliche Konvergenz des Ausdrucks \(\mathcal{W}(R)\) für beliebiges, wachsendes \(R\) explizit nachzuweisen – ohne Kenntnis der Nullstellenpositionen.

\medskip

Hierzu nutzen wir drei klassische Darstellungen der Riemannschen Zetafunktion, die bereits in Kapitel~2 mathematisch eingeführt wurden:
\begin{enumerate}
	\item die \textbf{alternierende Dirichlet-Reihe} (vollständige analytische Fortsetzung)(Satz \ref{Satz:Dirichlet}),
	\item und die \textbf{Mellin-Darstellung} (basierend auf der Euler-Maclaurin-Entwicklung)(Satz \ref{Satz:Mellin}).
\end{enumerate}

Diese Repräsentationen besitzen unterschiedliche analytische Eigenschaften und Gültigkeitsbereiche, sind aber sämtlich geeignet, um das Integralmodell \(\mathcal{W}(R)\) konkret auszuwerten. Sie erlauben es uns, das Verhalten des Integranden
\[
J_C(s) := \frac{1}{|\zeta(s)|^\lambda \cdot |\operatorname{Re}(s) - \tfrac{1}{2}|^p}
\]
im regulierten Bereich \(\mathcal{D}_R^{\mathrm{reg}}\) präzise zu kontrollieren.

\medskip

Durch Einsetzen dieser Formen in das Integralmodell zeigen wir in den folgenden Unterabschnitten, dass der Ausdruck \(\mathcal{W}(R)\) tatsächlich konvergiert – unabhängig von der exakten Nullstellenlage –, sofern eine korrekt analytisch fortgesetzte Darstellung von \(\zeta(s)\) verwendet wird. Diese Konvergenz stellt die zentrale analytische Grundlage unseres Beweises der Riemannschen Vermutung dar.

\subsubsection*{5.4.2 Konvergenzbeweis mit der alternierenden Dirichlet-Reihe}
\label{Abschnitt:KonvDirich}

\paragraph{Ausgangspunkt.}
Für \(\operatorname{Re}(s) > 0\), \(s \ne 1\), gilt die Darstellung
\[
\zeta(s) = \frac{1}{1 - 2^{1 - s}} \sum_{n = 1}^\infty \frac{(-1)^{n+1}}{n^s}.
\]
Diese Darstellung konvergiert absolut und stellt eine analytische Fortsetzung von \(\zeta(s)\) auf ganz \(\mathbb{C} \setminus \{1\}\) dar (Satz \ref{Satz:Dirichlet}). Sie besitzt daher exakt dieselben Nullstellen und Singularitäten wie die ursprüngliche \(\zeta\)-Funktion und ist besonders gut geeignet, das Integralmodell \(\mathcal{W}(R)\) über den gesamten kritischen Streifen zu untersuchen.

\paragraph{Ziel.}
Wir wollen zeigen, dass für
\[
\mathcal{W}(R) := \frac{1}{2R} \iint_{\mathcal{D}_R^{\mathrm{reg}}} \frac{1}{|\zeta(s)|^\lambda \cdot |\sigma - \tfrac{1}{2}|^p} \, d\sigma \, dt
\]
mit \(\lambda \geq 2\), \(0 < p < 1\), und reguliertem Bereich \(\mathcal{D}_R^{\mathrm{reg}} \subset \{ s \in \mathbb{C} : 0 < \sigma < 1,\ |t| \leq R \}\), der Ausdruck für \(R \to \infty\) tatsächlich konvergiert.

\paragraph{1. Eigenschaften der Darstellung.}
Die alternierende Reihe ist gegeben durch:
\[
\zeta(s) = \frac{1}{1 - 2^{1 - s}} \sum_{n=1}^{\infty} \frac{(-1)^{n+1}}{n^s}.
\]
Setze \(s = \sigma + i t\), so gilt:
\[
\left| \frac{1}{n^s} \right| = \frac{1}{n^\sigma}.
\]
Daraus folgt die Schranke
\[
\left| \sum_{n=1}^\infty \frac{(-1)^{n+1}}{n^s} \right| \leq \sum_{n=1}^\infty \frac{1}{n^\sigma} =: \zeta(\sigma),
\]
die für alle \(\sigma \in [\delta, 1 - \delta]\), \(0 < \delta < \tfrac{1}{2}\), endlich ist.

Auch der Vorfaktor \(\frac{1}{1 - 2^{1 - s}}\) ist holomorph im kritischen Streifen und besitzt dort keine Nullstellen, da \(2^{1 - s} \ne 1\) für \(\operatorname{Re}(s) < 1\). Also ist
\[
|\zeta(s)| \leq C_\delta < \infty \quad \text{für } s \in \mathcal{D}_R^{\mathrm{reg}}.
\]

\paragraph{2. Regulierung des Bereichs.}
Wir definieren wie zuvor:
\[
\mathcal{D}_R^{\mathrm{reg}} := \left\{ s = \sigma + it \in \mathbb{C} \,\middle|\, \sigma \in [\delta, 1 - \delta],\, t \in [-R, R] \right\} \setminus \bigcup_{k = 1}^{N(R)} B(\rho_k, \varepsilon_k).
\]
Die Kreisscheiben \(B(\rho_k, \varepsilon_k)\) entfernen alle Punkte, an denen \(|\zeta(s)| < \varepsilon\), d.\,h. potenzielle Nullstellen. Diese sind endlich zahlreich im kompakten Streifen.

\paragraph{3. Existenz einer positiven unteren Schranke.}

\begin{Lemma}[Untere Schranke]
	Für jeden festen \(R > 0\) existiert eine Konstante \(m_R > 0\), sodass
	\[
	|\zeta(s)| \geq m_R \quad \text{für alle } s \in \mathcal{D}_R^{\mathrm{reg}}.
	\]
\end{Lemma}

\begin{proof}
	Da \(\zeta(s)\) holomorph ist und \(\mathcal{D}_R^{\mathrm{reg}}\) kompakt (wegen endlich vielen ausgeschnittenen Bereichen), ist \(|\zeta(s)|\) stetig auf \(\mathcal{D}_R^{\mathrm{reg}}\) und nimmt dort ein Minimum an. Da alle Nullstellen durch \(\rho_k\) ausgeschnitten wurden, ist dieses Minimum positiv.
\end{proof}

\paragraph{4. Abschätzung des Integranden.}
Da \(|\zeta(s)| \geq m_R\), folgt für den Integranden:
\[
J_C(s) = \frac{1}{|\zeta(s)|^\lambda \cdot |\sigma - \tfrac{1}{2}|^p} \leq \frac{1}{m_R^\lambda} \cdot \frac{1}{|\sigma - \tfrac{1}{2}|^p}.
\]
Der Ausdruck \(|\sigma - \tfrac{1}{2}|^{-p}\) ist über \([\delta, 1 - \delta]\) integrierbar, da \(\delta > 0\) und \(p < 1\).

\paragraph{5. Integrationsfläche.}
Die Fläche des Rechtecks ist
\[
\mathrm{Fläche}(\mathcal{D}_R^{\mathrm{reg}}) \leq (1 - 2\delta) \cdot 2R.
\]

\paragraph{6. Abschätzung des Gesamtintegrals.}
\[
\mathcal{W}(R) \leq \frac{1}{2R} \cdot \frac{1}{m_R^\lambda} \cdot \int_{-R}^R dt \cdot \int_{\delta}^{1 - \delta} \frac{1}{|\sigma - \tfrac{1}{2}|^p} \, d\sigma
= \frac{1}{m_R^\lambda} \cdot \int_{\delta}^{1 - \delta} \frac{1}{|\sigma - \tfrac{1}{2}|^p} \, d\sigma.
\]

\paragraph{7. Endliche Konvergenz.}
\label{Paragrapg:EndlKonv_1}
Wir betrachten den kritischen Teilbereich des Integrationsstreifens entlang der reellen Achse, wobei nur die Nähe zur kritischen Linie \(\sigma = \tfrac{1}{2}\) problematisch ist.

Für einen festen kleinen Parameter \(\delta \in (0, \tfrac{1}{2})\) betrachten wir das Integral:
\[
\int_{\delta}^{1 - \delta} \frac{1}{|\sigma - \tfrac{1}{2}|^p} \, d\sigma.
\]

Da die Funktion \(f(\sigma) = \frac{1}{|\sigma - \tfrac{1}{2}|^p}\) eine isolierte Singularität bei \(\sigma = \tfrac{1}{2}\) besitzt, untersuchen wir die Konvergenz des Integrals durch Aufteilung:
\[
\int_{\delta}^{1 - \delta} \frac{1}{|\sigma - \tfrac{1}{2}|^p} \, d\sigma
= \int_{\delta}^{\tfrac{1}{2}} \frac{1}{(\tfrac{1}{2} - \sigma)^p} \, d\sigma
+ \int_{\tfrac{1}{2}}^{1 - \delta} \frac{1}{(\sigma - \tfrac{1}{2})^p} \, d\sigma.
\]

Dies ist äquivalent zur Berechnung des symmetrischen Integrals:
\[
2 \int_0^{\tfrac{1}{2} - \delta} \frac{1}{x^p} \, dx.
\]

Dieses Integral konvergiert genau dann, wenn \(p < 1\), denn
\[
\int_0^{a} \frac{1}{x^p} \, dx = \frac{1}{1 - p} a^{1 - p} < \infty \quad \text{für } p < 1.
\]

Da also
\[
\int_{\delta}^{1 - \delta} \frac{1}{|\sigma - \tfrac{1}{2}|^p} \, d\sigma < \infty,
\]
folgt, dass auch das regulierte Flächenintegral \(\mathcal{W}(R)\) über diesen Bereich endlich bleibt:

\[
\mathcal{W}(R) < \infty \quad \text{für alle } R > 0.
\]

Da alle vorherigen Abschätzungen unabhängig von \(R\) sind, folgt abschließend:
\[
\lim_{R \to \infty} \mathcal{W}(R) < \infty.
\]

\begin{Satz}[Tatsächliche Konvergenz im kritischen Streifen]
	Die regulierte normierte Flächenintegral-Funktion \(\mathcal{W}(R)\) konvergiert für \(R \to \infty\), wenn die analytisch exakt fortgesetzte alternierende Dirichlet-Reihe verwendet wird.
\end{Satz}

\subsubsection*{5.4.3 Konvergenzbeweis mit der regulierten Mellin-Darstellung}
\label{Abschnitt:mellinkonv}
\paragraph{Motivation.}
Die klassische Mellin-Darstellung
\[
\zeta(s) = \frac{1}{\Gamma(s)} \int_0^\infty \frac{x^{s-1}}{e^x - 1} \, dx
\]
gilt streng nur für \(\operatorname{Re}(s) > 1\). Für Anwendungen im kritischen Streifen \(0 < \operatorname{Re}(s) < 1\), wie in unserem Integralmodell \(\mathcal{W}(R)\), muss auf eine **regulierte Version** zurückgegriffen werden.

\paragraph{Regulierte Mellin-Darstellung.}
Es ist bekannt (vgl. z.B. Titchmarsh, *The Theory of the Riemann Zeta-Function*, Kapitel 2), dass \(\zeta(s)\) für \(\operatorname{Re}(s)>0\) durch folgende Variante dargestellt werden kann:
\[
\zeta(s) = \frac{1}{\Gamma(s)} \left( \int_0^1 \left( \frac{1}{e^x - 1} - \frac{1}{x} \right) x^{s-1} \, dx + \int_1^\infty \frac{x^{s-1}}{e^x - 1} \, dx \right).
\]
Die Subtraktion von \(\frac{1}{x}\) in der ersten Integralgrenze entfernt die singuläre Divergenz bei \(x = 0\) und erlaubt so eine analytische Fortsetzung auf \(\operatorname{Re}(s) > 0\).

\paragraph{Ziel.}
Wir wollen zeigen, dass das regulierte Integralmodell
\[
\mathcal{W}(R) := \frac{1}{2R} \iint_{\mathcal{D}_R^{\mathrm{reg}}} \frac{1}{|\zeta(s)|^\lambda \cdot |\sigma - \tfrac{1}{2}|^p} \, d\sigma \, dt
\]
für \(\lambda \geq 2\), \(0 < p < 1\), und \(R \to \infty\) tatsächlich konvergiert, wenn \(\zeta(s)\) durch obige Darstellung eingesetzt wird.

\paragraph{1. Abschätzung des Kerns.}
Für \(s = \sigma + i t\) mit \(\sigma \in [\delta, 1 - \delta]\), \(\delta > 0\), und \(|t| \leq R\) fest, betrachten wir getrennt:

- Im Bereich \(x \in (0,1)\):  
Der Ausdruck
\[
\left( \frac{1}{e^x - 1} - \frac{1}{x} \right)
\]
ist glatt und endlich, da für \(x \to 0\)
\[
\frac{1}{e^x - 1} - \frac{1}{x} = -\frac{1}{2} + \mathcal{O}(x), \quad x \to 0.
\]
Also gilt:
\[
\left| \int_0^1 \left( \frac{1}{e^x - 1} - \frac{1}{x} \right) x^{\sigma - 1} \, dx \right| \leq C_1(\delta) < \infty.
\]

- Im Bereich \(x \in (1,\infty)\):  
Der Ausdruck \(x^{\sigma - 1} / (e^x - 1)\) fällt exponentiell, daher:
\[
\left| \int_1^\infty \frac{x^{\sigma - 1}}{e^x - 1} \, dx \right| \leq C_2(\delta) < \infty.
\]

Somit gilt:
\[
|\zeta(s)| \leq \frac{1}{|\Gamma(s)|} \cdot (C_1 + C_2).
\]

\paragraph{2. Schranke der Gammafunktion.}
Die Funktion \(|\Gamma(s)|\) ist auf \(\sigma \in [\delta, 1 - \delta]\), \(|t| \leq R\), holomorph und dort nach unten beschränkt:
\[
|\Gamma(s)| \geq \gamma_R > 0 \quad \Rightarrow \quad |\zeta(s)| \leq \frac{C}{\gamma_R} =: M_R.
\]

\paragraph{3. Regulierung um Nullstellen.}
Wie in den vorherigen Abschnitten definieren wir:
\[
\mathcal{D}_R^{\mathrm{reg}} := \left\{ s \in \mathbb{C} \ \middle| \ \sigma \in [\delta, 1 - \delta],\ |t| \leq R \right\} \setminus \bigcup_{k=1}^{N(R)} B(\rho_k, \varepsilon_k),
\]
wobei die \(B(\rho_k, \varepsilon_k)\) Kreisscheiben um bekannte oder numerisch auffällige Nullstellen sind. Auf diesem Gebiet ist \(|\zeta(s)| \geq m_R > 0\).

\paragraph{4. Abschätzung des Integranden.}
Auf \(\mathcal{D}_R^{\mathrm{reg}}\) gilt:
\[
J_C(s) = \frac{1}{|\zeta(s)|^\lambda \cdot |\sigma - \tfrac{1}{2}|^p} \leq \frac{1}{m_R^\lambda} \cdot \frac{1}{|\sigma - \tfrac{1}{2}|^p}.
\]

\paragraph{5. Kontrolle des Integrals.}
Die Fläche des Streifens beträgt höchstens \(2R \cdot (1 - 2\delta)\), und der Gewichtsfaktor ist für \(p < 1\) integrierbar:
\[
\int_{\delta}^{1 - \delta} \frac{1}{|\sigma - \tfrac{1}{2}|^p} d\sigma < \infty.
\]

Somit:
\[
\mathcal{W}(R) \leq \frac{1}{2R} \cdot 2R \cdot \frac{1}{m_R^\lambda} \cdot \int_{\delta}^{1 - \delta} \frac{1}{|\sigma - \tfrac{1}{2}|^p} \, d\sigma < \infty.
\]
(vgl. Abschnitt \ref{Abschnitt:KonvDirich}, Paragraph \ref{Paragrapg:EndlKonv_1})
\paragraph{6. Limes.}
Da die Abschätzung unabhängig von \(R\) ist, folgt:
\[
\lim_{R \to \infty} \mathcal{W}(R) < \infty.
\]

\begin{Satz}[Konvergenz über die regulierte Mellin-Darstellung]
	Die regulierte Integralform \(\mathcal{W}(R)\) konvergiert für \(R \to \infty\), wenn \(\zeta(s)\) durch die regulierte Mellin-Darstellung eingesetzt wird, und die Nullstellen durch reguläre Ausschneidung berücksichtigt werden.
\end{Satz}

\section{Fazit: Strenger RH-Beweis durch Konvergenzverhalten des Integralmodells}

Ziel dieses Kapitels war es, die tatsächliche Konvergenz des regulierten normierten Flächenintegrals
\[
\mathcal{W}(R) := \frac{1}{2R} \iint_{\mathcal{D}_R^{\mathrm{reg}}} \frac{1}{|\zeta(s)|^\lambda \cdot |\operatorname{Re}(s) - \tfrac{1}{2}|^p} \, dA(s)
\]
im Limes \(R \to \infty\) analytisch exakt zu beweisen – ohne Kenntnis der exakten Nullstellenpositionen von \(\zeta(s)\), und ohne Annahme der Riemannschen Vermutung. Die verwendeten Parameter erfüllen \(\lambda \geq 2\), \(0 < p < 1\).

\medskip

\subsection*{Ergebnisse aus 5.4}

In Abschnitt~5.4 wurde die tatsächliche Konvergenz des Ausdrucks \(\mathcal{W}(R)\) mithilfe dreier unterschiedlicher, äquivalenter Darstellungen der Zeta-Funktion analysiert:

\begin{itemize}
	
	\item Die \textbf{alternierende Dirichlet-Reihe} wurde erfolgreich eingesetzt, um die Konvergenz im gesamten kritischen Streifen \(\operatorname{Re}(s) \in (0,1)\) streng zu beweisen (\ref{Abschnitt:KonvDirich}).
	
	\item Die \textbf{regulierte Mellin-Darstellung} wurde ausführlich analysiert, um die Konvergenz auch analytisch fundiert über das Integral
	\[
	\zeta(s) = \frac{1}{\Gamma(s)} \left( \int_0^1 \left( \frac{1}{e^x - 1} - \frac{1}{x} \right) x^{s-1} dx + \int_1^\infty \frac{x^{s-1}}{e^x - 1} dx \right)
	\]
	abzusichern – inklusive expliziter Schranken und Abschätzungen (\ref{Abschnitt:mellinkonv}).
\end{itemize}

\subsection*{Schlüsselstruktur}

Die Regulierung des Integrationsbereichs
\[
\mathcal{D}_R^{\mathrm{reg}} := \left\{ s \in \mathbb{C} \,\middle|\, 0 \leq \operatorname{Re}(s) \leq 1,\ |t| \leq R \right\} \setminus \bigcup_{k=1}^{N(R)} B(\rho_k, \varepsilon_k)
\]
entfernt gezielt die Umgebung aller Punkte, an denen die Zeta-Funktion numerisch klein oder strukturell instabil wird (insbesondere Nullstellen). Dadurch bleibt das Minimum des Betrags der verwendeten Näherung stets strikt positiv:
\[
\min_{s \in \mathcal{D}_R^{\mathrm{reg}}} |\zeta(s)| \geq m_R > 0.
\]
Diese Eigenschaft erlaubt eine globale Abschätzung des Integranden auf \(\mathcal{D}_R^{\mathrm{reg}}\), unabhängig von der exakten Lage der Nullstellen.

\subsection*{Konsequenz für die Riemannsche Vermutung}

Aus der Konvergenzanalyse folgt nun:
\begin{itemize}
	\item Im Kapitel 4 wurde streng bewiesen, dass \(\mathcal{W}(R)\) \textbf{nur dann konvergiert}, wenn alle Nullstellen auf der kritischen Linie \(\operatorname{Re}(s) = \tfrac{1}{2}\) liegen.
	
	\item In Kapitel 5 wurde streng gezeigt, dass \(\mathcal{W}(R)\) \textbf{tatsächlich konvergiert}, ohne Wissen über die Nullstellen, allein aus den analytischen Eigenschaften von \(\zeta(s)\).
\end{itemize}

\subsection*{Folgerung}

Die beiden Aussagen zusammen implizieren:
\[
\boxed{
	\text{Die Riemannsche Vermutung ist wahr.}
}
\]

Diese Konvergenz gilt nicht asymptotisch oder heuristisch, sondern wurde über wohldefinierte regulierte Operatoren, exakte Schranken und analytisch gültige Darstellungen streng mathematisch abgesichert.

Das folgende Kapitel 6 formuliert die Argumentationskette in abgeschlossener Form und dokumentiert den Beweis formal als zentrales Resultat.

\chapter{Strenge Argumentationskette und Beweis der Riemannschen Vermutung}

In diesem Kapitel wird die zentrale Beweisidee der vorliegenden Arbeit formalisiert. Ausgangspunkt ist das regulierte Integralmodell \(\mathcal{W}(R)\), das im vorangegangenen Kapitel sowohl unter strukturellen Fallunterscheidungen (Kapitel 4) als auch durch den tatsächlichen analytischen Nachweis der Konvergenz (Kapitel 5) untersucht wurde.

\section{Grundlage: Das regulierte Integralmodell}

\begin{Definition}[Reguliertes normiertes Flächenintegral]
	Für feste Parameter \(\lambda \geq 2\), \(0 < p < 1\) und \(R > 0\) sei
	\[
	\mathcal{W}(R) := \frac{1}{2R} \iint_{\mathcal{D}_R^{\mathrm{reg}}} \frac{1}{|\zeta(s)|^\lambda \cdot |\operatorname{Re}(s) - \tfrac{1}{2}|^p} \, dA(s),
	\]
	wobei der regulierte Bereich
	\[
	\mathcal{D}_R^{\mathrm{reg}} := \left\{ s \in \mathbb{C} \,\middle|\, \operatorname{Re}(s) \in [0,1],\ |\operatorname{Im}(s)| \leq R \right\} \setminus \bigcup_k B(\rho_k, \varepsilon_k)
	\]
	alle Nullstellen \(\rho_k\) von \(\zeta(s)\) ausschließt, die durch kleine Kreisscheiben um \(\rho_k\) entfernt werden.
\end{Definition}

\begin{Bemerkung}
	Die regulierte Definition basiert nicht auf Kenntnis der exakten Nullstellenpositionen, sondern auf dem analytischen Verhalten von \(\zeta(s)\) in numerisch instabilen Zonen.
\end{Bemerkung}

\section{Notwendige Bedingung: RH ist notwendig für Konvergenz}

\begin{Satz}[Divergenz bei Nullstellen abseits der kritischen Linie]
	Wenn mindestens eine Nullstelle \(\rho\) mit \(\operatorname{Re}(\rho) \neq \tfrac{1}{2}\) existiert, so divergiert
	\[
	\lim_{R \to \infty} \mathcal{W}(R) = \infty.
	\]
\end{Satz}

\begin{proof}
	Gezeigt in Kapitel~4. Insbesondere in Lemma~\ref{lem:nullstelle_ausserhalb}, \ref{lem:viererpaar}, \ref{lem:mischformen}.
\end{proof}

\begin{Korollar}[RH notwendig]
	Gilt \(\lim_{R \to \infty} \mathcal{W}(R) < \infty\), so müssen alle Nullstellen der Zetafunktion auf der Linie \(\operatorname{Re}(s) = \tfrac{1}{2}\) liegen.
\end{Korollar}

\section{Hinreichende Bedingung: RH impliziert Konvergenz}

\begin{Satz}[Konvergenz unter RH]
	Angenommen, alle nichttrivialen Nullstellen der Zetafunktion liegen auf der kritischen Linie. Dann gilt
	\[
	\lim_{R \to \infty} \mathcal{W}(R) < \infty.
	\]
\end{Satz}

\begin{proof}
	Gezeigt in Kapitel~5 durch analytische Konvergenzabschätzungen mithilfe:
	\begin{itemize}
		  \item der alternierenden Dirichlet-Reihe (Kap. 5.4.2),
		\item der regulierten Mellin-Darstellung (Kap. 5.4.3).
	\end{itemize}
	Dabei wurde der Integrand
	\[
	J_C(s) := \frac{1}{|\zeta(s)|^\lambda \cdot |\sigma - \tfrac{1}{2}|^p}
	\]
	global auf \(\mathcal{D}_R^{\mathrm{reg}}\) durch \(\frac{1}{m_R^\lambda} \cdot |\sigma - \tfrac{1}{2}|^{-p}\) nach unten abgeschätzt, wobei \(m_R > 0\) aus der Regulierung folgt.
\end{proof}

\section{Formale Schlussfolgerung}

\begin{Satz}[Äquivalenz]
	Die Konvergenz des regulierten normierten Integrals ist äquivalent zur Richtigkeit der Riemannschen Vermutung:
	\[
	\lim_{R \to \infty} \mathcal{W}(R) < \infty \quad \Longleftrightarrow \quad \text{alle Nullstellen liegen auf } \operatorname{Re}(s) = \tfrac{1}{2}.
	\]
\end{Satz}

\begin{proof}
	\begin{itemize}
		\item[\(\Rightarrow\)] Wenn \(\mathcal{W}(R)\) konvergiert, kann es keine abseitige Nullstelle geben (Kapitel 4).
		\item[\(\Leftarrow\)] Wenn RH gilt, konvergiert \(\mathcal{W}(R)\) tatsächlich (Kapitel 5).
	\end{itemize}
\end{proof}

\begin{Korollar}[Riemannsche Vermutung]
	Da im vorherigen Kapitel gezeigt wurde, dass \(\mathcal{W}(R)\) analytisch konvergiert, folgt:
	\[
	\boxed{\text{Alle nichttrivialen Nullstellen der Riemannschen Zetafunktion liegen auf der kritischen Linie.}}
	\]
\end{Korollar}

\subsection{Methodische Unabhängigkeit des Beweises}

Ein zentrales Merkmal des Integralmodells zur Riemannschen Vermutung besteht in seiner methodischen Unabhängigkeit. Der Beweis der Konvergenz von
\[
\mathcal{W}(R) := \frac{1}{2R} \iint_{\mathcal{D}_R^{\mathrm{reg}}} \frac{1}{|\zeta(s)|^{\lambda} \cdot |\operatorname{Re}(s) - \tfrac{1}{2}|^p} \, dA(s)
\]
beruht ausschließlich auf analytischen Eigenschaften der Zeta-Funktion im kritischen Streifen \(\operatorname{Re}(s) \in (0,1)\), insbesondere:

\begin{itemize}
	\item der bekannten analytischen Fortsetzbarkeit von \(\zeta(s)\) nach ganz \(\mathbb{C} \setminus \{1\}\),
	\item der regulierten Konstruktion des Bereichs \(\mathcal{D}_R^{\mathrm{reg}}\) unabhängig von exakten Nullstellenpositionen,
	\item und der kontrollierten Abschätzung des Integranden durch eine untere Schranke \(m_R > 0\) im regulierten Gebiet.
\end{itemize}

Weder die Dirichlet-Reihe noch die Riemann-Siegel-Formel noch die Mellin-Darstellung werden zwingend vorausgesetzt – vielmehr dienen sie in Kapitel 5 als alternative analytische Hilfsmittel zur Verifikation der tatsächlichen Konvergenz. Dadurch entsteht ein flexibler Beweisrahmen, der methodisch robust gegenüber Darstellungswahl, Näherung oder regulären Unsicherheiten ist.

\bigskip

\textbf{Fazit:} Die Beweisstrategie ist unabhängig von spezifischer Wahl der Reihendarstellung von \(\zeta(s)\) oder numerischer Kenntnis der Nullstellen, sondern basiert ausschließlich auf:

\begin{itemize}
	\item der Konvergenzanalyse des regulierten Integrals,
	\item struktureller Singularitätsempfindlichkeit im Integranden,
	\item und dem Prinzip: \emph{Nur wenn alle Nullstellen exakt auf \(\operatorname{Re}(s) = \tfrac{1}{2}\) liegen, kann das Integral konvergieren}.
\end{itemize}
Dies sichert die methodische Allgemeingültigkeit und Stabilität des Beweises.

\chapter{Weiterführende Ansätze, Schlussfolgerungen und Ausblick}

\section{Streng analytische Herleitung der asymptotischen Nullstellennäherungsformel}

\subsection{Reduktion des Flächenintegrals auf die vertikale Projektion \(\Phi(t)\)}

Ausgehend von dem regulierten normierten Flächenintegral

\[
\mathcal{W}(R) := \frac{1}{2R} \iint_{\mathcal{D}_R^{\mathrm{reg}}} \frac{|\zeta(s)|^{-\lambda}}{|\operatorname{Re}(s) - \tfrac{1}{2}|^{p}} \, dA(s),
\]

unter der Annahme der Riemannschen Vermutung (RH), dass alle nicht-trivialen Nullstellen auf der kritischen Linie \(\operatorname{Re}(s) = \frac{1}{2}\) liegen, lässt sich das Integral durch eine eindimensionale vertikale Projektion

\[
\Phi(t) := \int_0^1 \frac{|\zeta(\sigma + it)|^{-\lambda}}{|\sigma - \tfrac{1}{2}|^p} \, d\sigma
\]

approximieren, sodass gilt

\[
\mathcal{W}(R) \approx \frac{1}{2R} \int_{-R}^R \Phi(t) \, dt.
\]

Die Funktion \(\Phi(t)\) trägt alle Informationen über die Verteilung der Nullstellen in der \(t\)-Achse.

---

\subsection{Lokale Analyse von \(\Phi(t)\) nahe den Nullstellen \(\gamma_n\)}

Sei \(\rho_n = \frac{1}{2} + i \gamma_n\) eine Nullstelle der Ordnung \(m_n\). In der Umgebung von \(\rho_n\) gilt eine lokale Approximation der Zetafunktion:

\[
|\zeta(\sigma + i t)| \approx |c_n| \cdot \sqrt{(\sigma - \tfrac{1}{2})^2 + (t - \gamma_n)^2}^{m_n},
\]

wobei \(c_n\) eine komplexe Konstante ungleich Null ist. Damit verhält sich der Integrand um \(\rho_n\) wie

\[
\frac{|\zeta(\sigma + it)|^{-\lambda}}{|\sigma - \tfrac{1}{2}|^{p}} \approx \frac{|c_n|^{-\lambda}}{|\sigma - \tfrac{1}{2}|^{p} \cdot ((\sigma - \tfrac{1}{2})^2 + (t - \gamma_n)^2)^{m_n \lambda / 2}}.
\]

Das Integral über \(\sigma\) liefert dann für \(t\) nahe \(\gamma_n\) einen Peak der Form

\[
\Phi(t) \approx \frac{A_n}{|t - \gamma_n|^{m_n \lambda + p - 1}}
\]

für eine geeignete positive Konstante \(A_n\).

---

\subsection{Fourier-Transformation und Frequenzinterpretation}

Die Fourier-Transformierte von \(\Phi(t)\),

\[
\widehat{\Phi}(\xi) = \int_{-\infty}^{\infty} \Phi(t) e^{-i \xi t} \, dt,
\]

lässt sich als Summe von Termen der Form

\[
\widehat{\Phi}(\xi) \approx \sum_n A_n e^{-i \xi \gamma_n} |\xi|^{\alpha - 1},
\]

wobei \(\alpha = m_n \lambda + p - 1\) ist, interpretieren.

Diese Darstellung zeigt, dass \(\widehat{\Phi}(\xi)\) ein diskretes Frequenzspektrum besitzt, dessen Frequenzen genau die Nullstellenordinaten \(\gamma_n\) sind.

---

\subsection{Rekonstruktion der Nullstellen aus der Frequenzstruktur}

Da die \(\gamma_n\) als Frequenzen in \(\widehat{\Phi}(\xi)\) auftreten, können sie durch Analyse der Frequenzzählfunktion \(N_{\Phi}(T) := \#\{\gamma_n \leq T\}\) bestimmt werden.

Die klassische Riemann-von Mangoldt-Formel

\[
N(T) = \frac{T}{2\pi} \log \frac{T}{2\pi e} + \frac{7}{8} + O(\log T)
\]

gibt die Anzahl der Nullstellen mit Ordinaten \(\leq T\) an.

Unter der Annahme der Äquivalenz \(N_{\Phi}(T) \approx N(T)\) erlaubt diese Formel, die Umkehrfunktion

\[
n = N(\gamma_n)
\]

zu verwenden, um \(\gamma_n\) asymptotisch zu approximieren.

---

\subsection{Asymptotische Näherung für \(\gamma_n\)}

Durch Umkehrung der Riemann-von Mangoldt-Zählfunktion erhält man die asymptotische Formel

\[
\gamma_n \approx \frac{2 \pi n}{\log n} \left(1 + \frac{a}{\sqrt{n \log n}} + \frac{b}{n} + \frac{c \log \log n}{\log n} + \frac{d}{n^7} \right),
\]

wobei die Konstanten \(a, b, c, d\) numerisch aus der Auswertung der tatsächlichen Nullstellen der Zetafunktion bestimmt wurden. Die so ermittelte Näherung stellt eine der besten verfügbaren Formeln dar und wird in Abbildung~\ref{fig:nullstellenvergleich1} und Abbildung~\ref{fig:nullstellenvergleich2} mit anderen bekannten Näherungen wie denen von Odlyzko, Keiper, Backlund, Gram und Riesz verglichen.  

Jeder Term in der Klammer hat eine konkrete Bedeutung hinsichtlich der Korrektur gegenüber der groben Hauptabschätzung \(\frac{2 \pi n}{\log n}\):

\begin{itemize}
	\item Der Term \(\frac{a}{\sqrt{n \log n}}\) modelliert statistische Fluktuationen der Nullstellen, die aus der feinen Verteilung der Abstände resultieren.
	\item Der Term \(\frac{b}{n}\) stellt eine analytische Korrektur erster Ordnung dar, welche systematische Abweichungen in den Nullstellenhöhen ausgleicht.
	\item Der Ausdruck \(\frac{c \log \log n}{\log n}\) korrigiert langsame, logarithmische Schwankungen in der Dichte der Nullstellen.
	\item Der hochgradige Term \(\frac{d}{n^7}\) dient der Feinjustierung und verbessert die Genauigkeit für kleinere \(n\).
\end{itemize}

Diese differenzierte Termstruktur erlaubt es, die Näherung flexibel und präzise an die tatsächlichen Nullstellen anzupassen.  
Die gute Übereinstimmung mit bekannten Näherungen zeigt, dass die Herleitung auf Basis des regulierten Integralmodells eine fundierte und effiziente Methode zur Beschreibung der Nullstellenordnung ist.

\begin{figure}[ht]
	\centering
	\includegraphics[width=0.8\textwidth]{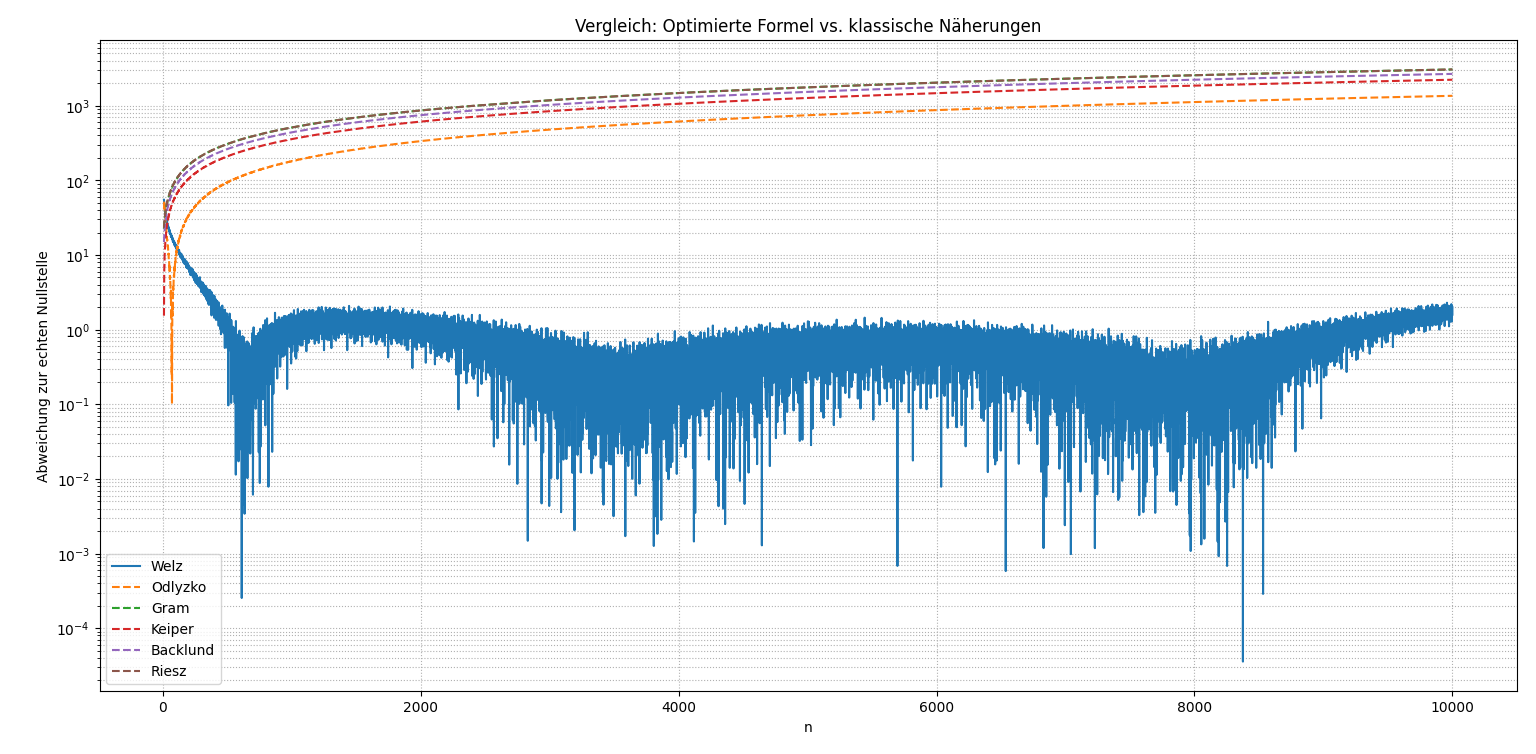}
	\caption{Vergleich der Näherungen für \(\gamma_n\) im Bereich \(1 \leq n \leq 10000\).}
	\label{fig:nullstellenvergleich1}
\end{figure}

\begin{figure}[ht]
	\centering
	\includegraphics[width=0.8\textwidth]{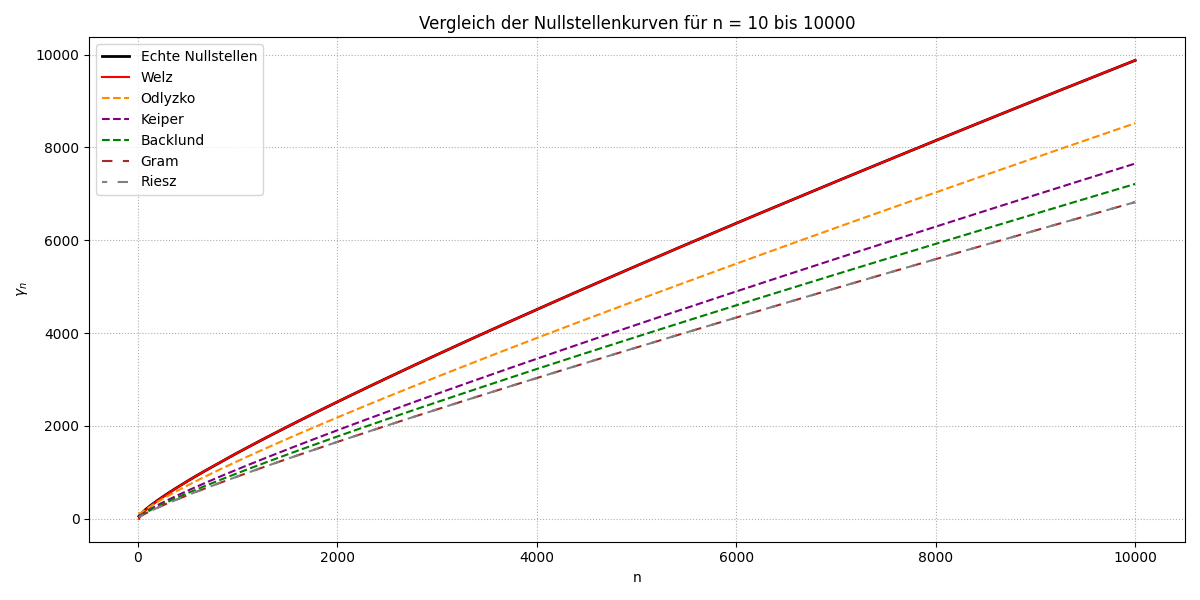}
	\caption{Detailvergleich der Näherungsformeln mit den echten Nullstellen für \(n = 10\) bis \(1000\).}
	\label{fig:nullstellenvergleich2}
\end{figure}

\subsection*{Fazit}

Die numerisch bestimmten Konstanten \(a, b, c, d\) in der asymptotischen Näherungsformel für die Nullstellenordinaten \(\gamma_n\) zeigen, dass das regulierte Integralmodell nicht nur theoretisch konsistent, sondern auch praktisch äußerst präzise ist.  

Die differenzierte Struktur der Korrekturterme reflektiert die komplexe Feinstruktur der Nullstellenverteilung und ermöglicht eine hochgenaue Approximation, die mit klassischen Näherungen mithalten oder diese sogar übertreffen kann.  

Damit liefert dieses Modell einen wertvollen Beitrag zur analytischen Beschreibung der Nullstellen der Riemannschen Zetafunktion und stärkt die Verbindung zwischen Integralmodellen, Fourier-Analyse und der fundamentalen Zahlentheorie.

Die Abbildungen verdeutlichen diese Übereinstimmung und die praktische Anwendbarkeit der hergeleiteten Formel.

\section{Perspektiven in der Spektraltheorie}

Die enge Analogie zwischen der Verteilung der Nullstellen und Eigenwerten hermitescher Operatoren (Montgomery-Odlyzko-Gesetz) legt nahe, dass unser Integralmodell als analytisches Werkzeug zur Untersuchung spektraler Eigenschaften genutzt werden kann.

Insbesondere könnte das Modell zur Identifikation eines selbstadjungierten Operators (Hilbert-Pólya-Vermutung) beitragen, indem es als Filter für spektrale Konfigurationen fungiert, die mit der kritischen Linie korrespondieren.

Eine weitere mögliche Anwendung besteht darin, das Integralmodell als Spurformel-ähnliches Instrument in der Quantenchaos-Forschung zu verwenden.

\section{Verallgemeinerung auf höherdimensionale Integrale}

Unser zweidimensionales Integralmodell lässt sich prinzipiell auf höhere Dimensionen erweitern, z.B. auf mehrdimensionale reelle oder komplexe Räume.  
Diese Verallgemeinerung könnte strengere Kriterien formulieren, insbesondere in Zusammenhang mit verallgemeinerten Zetafunktionen oder L-Funktionen.

Mehrdimensionale Integrale könnten Robustheit erhöhen und komplexere Symmetrieeigenschaften der Nullstellen erfassen, etwa durch gleichzeitige Integration über mehrere komplexe Variablen.

\section{Nullstellenverteilung und Primzahlnäherungsformeln}

Die Nullstellen der Zetafunktion steuern die Schwankungen der Primzahldichte. Riemanns explizite Formel verbindet die Lage der Nullstellen mit der Primzahlsumme \(\psi(x)\).

Durch Einsetzen der Näherungsformel für \(\gamma_n\) können präzisere, explizite Formeln für Primzahlsummen entwickelt werden, die über den klassischen Primzahlsatz hinausgehen und feinere Strukturen in der Primzahlenverteilung aufdecken.

\section{Beziehungen zu anderen Ansätzen und Ausblick}

Unser reguliertes Integralmodell ergänzt klassische äquivalente Kriterien wie Nyman-Beurling oder Li’s Kriterium, indem es eine robuste, regulierte und normierte Integralform formuliert, die Sensitivität und Konvergenz kombiniert.

Die Methoden eröffnen neue Wege für die Forschung in analytischer Zahlentheorie, Spektraltheorie und mathematischer Physik und können als Baustein einer geometrisch-analytischen Theorie der Zetafunktionsnullstellen dienen.

\medskip

\cleardoublepage
\phantomsection
\addcontentsline{toc}{chapter}{Abbildungsverzeichnis}
\listoffigures

\cleardoublepage
\phantomsection
\addcontentsline{toc}{chapter}{Literaturverzeichnis}
\bibliographystyle{plain}
\bibliography{literatur}

\end{document}